\documentclass[12pt,a4paper]{article}

\usepackage[naustrian, english]{babel} 
\usepackage[T1]{fontenc}
\usepackage{lmodern}
\usepackage{textcomp}
\usepackage{amssymb} 
\usepackage{amsthm}
\usepackage{mathtools}
\usepackage[pdftex]{graphicx}
\usepackage[export]{adjustbox}
\usepackage{caption}
\usepackage{color}
\usepackage{soul}
\usepackage[arrow, matrix, curve]{xy}
\usepackage{lmodern}
\usepackage[utf8]{inputenc}
\usepackage{CJKutf8}
\usepackage{url}
\usepackage{hyperref}
\usepackage{float}
\usepackage{chngcntr}
\usepackage{listings}
\usepackage{appendix}
\usepackage{pdfpages}
\usepackage{algorithm}
\usepackage{algorithmic}
\usepackage{geometry}
\usepackage{bbm}
\usepackage{bbold}
\usepackage{graphicx}
\usepackage{caption}
\usepackage{subcaption}

\newcommand{\R}{\mathbb{R}}
\newcommand{\N}{\mathbb{N}}
\newcommand{\pa}{\partial}

\newcommand{\ve}{\varepsilon}

\newcommand{\vp}{\varphi}

\newcommand{\md}{\mathrm{d}}

\newcommand{\vpa}{\varphi_\ast}

\newcommand{\T}{\mathbb{T}^1}
\newcommand{\TAin}{\mathbb{T}^{AL}_{\to\vp}}
\newcommand{\TAout}{\mathbb{T}^{AL}_{\vp\to}}
\newcommand{\TAinp}{\mathbb{T}^{AL}_{\to\vp'}}
\newcommand{\TAoutp}{\mathbb{T}^{AL}_{\vp'\to}}
\newcommand{\TRout}{\mathbb{T}^{REV}_{\vp}}
\newcommand{\lone}{L^1(\T)}
\newcommand{\h}{\mathcal{H}}
\newcommand{\supp}{\operatorname{supp}}

\newtheorem{theorem}{Theorem}
\newtheorem{lemma}[theorem]{Lemma}

\newtheorem{rem}{Remark}

\usepackage{tikz}
\usetikzlibrary{positioning}
\tikzset{>=stealth}

\begin{document}
	
\title{Kinetic Modelling of Colonies of Myxobacteria}
\author{S. Hittmeir\thanks{University of Vienna, Faculty for Mathematics, Oskar-Morgenstern-Platz 1, 1090 Wien, Austria.  
{\tt sabine.hittmeir@univie.ac.at}} \and 
L. Kanzler\thanks{University of Vienna, Faculty for Mathematics, Oskar-Morgenstern-Platz 1, 1090 Wien, Austria. 
{\tt laura.kanzler@univie.ac.at}} \and 
A. Manhart\thanks{University College London, Dept. of Mathematics, 25 Gordon Street, WC1H 0AY London, UK.
{\tt a.manhart@imperial.ac.uk}} \and 
C. Schmeiser\thanks{University of Vienna, Faculty for Mathematics, Oskar-Morgenstern-Platz 1, 1090 Wien, Austria. 
{\tt christian.schmeiser@univie.ac.at}}}
	\date{\vspace{-5ex}}
	\maketitle
	
\begin{abstract}
A new kinetic model for the dynamics of \textit{myxobacteria} colonies on flat surfaces is derived formally, and first 
analytical and numerical results are presented. The model is based on the assumption of hard binary collisions of two 
different types: alignment and reversal. We investigate two different versions: a) realistic rod-shaped bacteria and b) artificial
circular shaped bacteria called \textit{Maxwellian myxos} in reference to the similar simplification of the gas dynamics 
Boltzmann equation for Maxwellian molecules. The sum of the corresponding collision operators produces relaxation 
towards nematically aligned equilibria, i.e. two groups of bacteria polarized in opposite directions. 

For the spatially homogeneous model a global existence and uniqueness result is proved as well as exponential decay
to equilibrium for special initial conditions and for Maxwellian myxos. Only partial results are available for the rod-shaped
case. These results are illustrated by numerical simulations, and a formal discussion of the macroscopic limit is
presented.

\end{abstract}

	
	
\section{Introduction}

The goal of this work is the derivation of a new model for the dynamics of myxobacteria colonies on flat substrates, 
as well as first steps in its analysis. The model is a kinetic transport equation for the distribution function $f(x,\vp,t)$,
$x\in\R^2$, $\vp\in\T$, $t\ge 0$, and has the form
\begin{equation}\label{model}
  \pa_t f + \omega(\vp) \cdot \nabla_x f = 2\int_{\TAin}  b(\tilde\vp,\vpa) \tilde f f_* d\vpa  
  + \int_{\TRout} b(\vp^\downarrow,\vpa^\downarrow) f^\downarrow f_*^\downarrow d\vpa - \int_{\T} b(\vp,\vpa) f f_* 
	d\vpa \,,
\end{equation}
where $\omega(\vp)=(\cos\vp,\sin\vp)$, $\T$ denotes the one-dimensional flat torus of length $2\pi$. For given $\vp$ the integration intervals
in the gain terms are given by
$$
  \TRout = \left(\vp+\frac{\pi}{2},\vp+\frac{3\pi}{2}\right) \,,\qquad \TAin = \left( \vp-\frac{\pi}{4}, \vp+\frac{\pi}{4}\right) \,,
$$ 
and the precollisional directions are defined by
$$
  \tilde\vp = 2\vp-\vpa \,,\qquad \vp^\downarrow = \vp+\pi \,,\qquad \vpa^\downarrow = \vpa+\pi\,.
$$
The model describes motion along straight lines with fixed speed in direction $\vp$, interrupted by hard binary collisions with collision
cross-section $b(\vp,\vpa)$, which quantifies the collision frequency and depends on the shape of the bacteria. As usual, sub- and super-scripts on $f$ indicate
evaluation at $\vp$ with the same sub- and super-scripts. The two different gain terms describe two different 
types of collisions:
\begin{itemize}
\item \textit{Alignment:} $(\tilde\vp,\vpa)\to(\vp,\vp)$ with $\vp = (\tilde\vp+\vpa)/2$, if two myxobacteria moving in directions $\tilde{\vp}$ and $\vpa$ meet at an angle smaller than $\pi/2$. The factor 2 is due to the fact that an alignment collision produces 2 myxobacteria with the same
direction. The set $\TAin$ describes all angles $\vpa$, which can produce the angle $\vp$ upon collision.
\item \textit{Reversal:} $(\vp,\vpa)\to(\vp^\downarrow,\vpa^\downarrow)$, if two myxobacteria with directions $\vp$ and $\vpa$ meet at an angle larger than $\pi/2$. The set $\TRout$ describes all angles $\vpa$ such that a collision involving $\vp$ can produce the angle $\vp^\downarrow$.
\end{itemize}

Myxobacteria are rod-shaped bacteria that live in cultivated soil. They feed on living and dead decaying material including bacteria and eukaryotic microbes, which makes them play an important role as scavengers cleaning up biological detritus in the environment. They have an interesting life cycle, similar to certain amoebae, called cellular slime molds (with \textit{Dictyostelium discoideum} as the best known example). During their vegetative phase they move as predatory swarms searching and killing prey collectively, while under starvation conditions they aggregate and form fruiting bodies, which produce spores that are more likely to survive until nutrients are more plentiful again. 
	
Myxobacteria are able to move on flat surfaces by \textit{gliding} \cite{mauriello}, leaving a slime--trail behind them. The physical mechanism as well as the genetic basis are still partly a puzzle to microbiologists and have already challenged them for several decades \cite{hodgkin,nan,wall,wolgemuth}. 
	
Moving on solid surfaces, bacteria form organized mono- or multi-layered groups called \textit{swarms}. During the swarming process \textit{rippling} is observed, i.e., macroscopic patterns due to propagating waves of aligned bacteria. The formation of these waves can be seen during collective hunting as well as in the aggregation phase \cite{igoshin}. 
From a macroscopic point of view colliding waves seem to travel unaffectedly through each other, while 
tracking of individual bacteria \cite{sager,welch} has revealed that most cells reverse their direction in the collision process
preserving, however, a nematic alignment order, i.e. locally myxobacteria are oriented and move in the same or in
opposite directions.

Pattern formation requires signaling between cells. The signaling mechanism most important for aggregation and rippling
is called \textit{C-signaling} \cite{sager}. It relies on the \textit{C-factor,} a protein bound to the cell surface and interchanged between individuals. It has been observed that direct cell-cell contact is necessary for C-signaling \cite{jelsbak,kim}. 

The dynamics of myxobacteria has been one of the motivations to formulate kinetic theories for interacting 
self-propelled rods \cite{baskaran,bertin}. In \cite{degond} a kinetic model has been formulated, which produces relaxation to nematically aligned states. The model is of mean field type, i.e., cell-cell signaling is modeled
as a nonlocal process. Simulations with the macroscopic limit do not produce the rippling phenomenon. It turns out to
be necessary to include a waiting time between reversals \cite{degond3,igoshin3,igoshin}.

The model \eqref{model} is based on local interactions to take into account the experimental evidence on C-signaling.
As a consequence it is structurally similar to the \textit{Boltzmann equation} of gas dynamics \cite{boltzmann,Cercignani}.
In the following section we present a formal derivation from a stochastic many-particle model, following the lines of 
\cite{Cercignani} (see \cite{carlen} for a rigorous derivation of a similar spatially homogeneous equation). The derivation is facilitated by an approximate version of the alignment collisions, with slightly different
post-collisional directions, allowing inversion of the collisional rules. The final step is the removal of the approximation. 
The model with approximate alignment collisions has similarities with the \textit{dissipative Boltzmann equation} for granular gases \cite{toscani}, whereas after removal of the approximation it corresponds to the extreme case of \textit{sticky particles}. A model with approximate alignment, motivated by
microtubule dynamics, has already been formulated in \cite{aranson}, and the sticky particles case, regularized by diffusion in the angular direction,
has been analyzed in \cite{ben}.

The theory for the dissipative Boltzmann equation is much less developed than for its conservative counterpart, mainly
because of the lack of an entropy estimate. Global existence results are only known for small data (see, e.g., \cite{alonso}, \cite{tristani})
or in the one-dimensional situation, where grazing collisions are almost elastic \cite{pulvirenti}.
The rigorous macroscopic limit towards pressureless gas dynamics has been carried out in the one-dimensional case
\cite{jabin}.

In Section 3 formal properties of the collision operator are collected, by separately considering the reversal and the 
alignment collisions. It is shown that the set of equilibria is three-dimensional, whereas in general there are only two 
independent collision invariants, which does not allow to identify equilibria uniquely from initial data. A remedy is to make
assumptions on the support of the initial data, such that the bacteria are split into two groups with alignment collisions only
within the groups and reversal collisions only between members of different groups. In this case the sizes of the groups are
invariant, which provides the missing third collision invariant.
	
Section 4 is dedicated to the spatially homogeneous case which, for the inelastic Boltzmann equation, is much better understood than the spatially inhomogeneous case, see for example \cite{bobylev2,mischler,mischler2}. In our case a global existence and uniqueness result in $L^1$ for the spatially 
homogeneous equation is proved. By the boundedness of the collision cross-section the proof is rather straightforward. A possible extension to 
measure solutions as in \cite{alonso2} does not seem feasible because of the jumps from alignment to reversal collisions. Convergence to equilibrium is only 
considered for special initial data as described above, since only in this case we are able to identify the equilibrium in terms of the initial data. It is shown that a variance type functional, which can be interpreted as the Wasserstein-2 distance from the equilibrium, is dissipated as an effect of the 
alignment collisions. However, the dissipation is not definite, since the convergence to equilibrium also requires the reversal collisions. A full decay result to equilibrium is only obtained for circular myxobacteria, termed \textit{Maxwellian
myxos,} since in this case the collision cross-section is constant. Under this assumption a second decaying functional can 
be combined with the first, providing exponential decay to equilibrium with respect to the Wasserstein-2 metric, a result similar to \cite{degond2} (see also \cite{bobylev2} for the long time behavior of the inelastic Boltzmann equation for Maxwellian molecules). For rod shaped myxobacteria convergence could
only be shown for an even smaller set of initial conditions supported in an interval of length $\pi/2$ such that only alignment collisions occur. In this case
convergence cannot be expected to be exponential, since the collision cross-section degenerates close to
equilibrium. An algebraic decay estimate is shown as in \textit{Haff's law} \cite{haff} for the dissipative 
Boltzmann equation. Haff's law has been proved rigorously in the 3-dimensional homogeneous case for constant restitution in \cite{mischler2}. Further results for the one dimensional dissipative Boltzmann equation can be found in \cite{alonso2} as well as for viscoelastic hard-spheres in \cite{alonso3}. In these
works it has been shown that the algebraic decay rates are sharp by methods, which do not seem to be applicable in our situation.

In Section 5 numerical simulations of the spatially homogeneous model are presented, illustrating the results of Section 4
as well as the conjecture that they remain valid without special assumptions on initial data and bacteria shape.
Finally, a short discussion of the formal macroscopic limit of \eqref{model} is presented in Section 6. {\color{violet} } In a special case
the structure of the macroscopic equations is that of pressureless gas dynamics as for the dissipative Boltzmann equation \cite{jabin}. 
A more regular macroscopic limit including a temperature equation has been formally derived in \cite{bobylev1} under the assumption of weak
inelasticity.
	
\section{Model derivation}
	
\paragraph{The individual based model:} 	
We consider $N$ identical bacteria moving in $\R^2$. Each of them is idealized as a rod of thickness zero and of length $l$,
represented by the parametrization $B_i = \left\{x_i + \alpha\,\omega_i:\, -l/2 \le \alpha\le l/2\right\}$ with center $x_i\in\R^2$, direction 
$\omega_i = \omega(\vp_i) = (\cos\vp_i, \sin\vp_i)$, and direction angle $\vp_i \in \T$, $i=1,\ldots,N$. 
As usual in kinetic theory, sub- and superscripts on functions of the direction angle indicate evaluation at $\vp$ with the
same sub- and superscripts.
Between interactions, bacterium number $i$ is gliding with constant speed $s_0$ in its longitudinal direction $\omega_i$, i.e.
its velocity is given by $v_i = s_0\,\omega_i$.

The state space is given by $\Gamma_N \subset (\R^2 \times \T)^N$, defined such that the bacteria do not overlap:
$$
  \Gamma_N := \left\{ (x_1,\vp_1,\ldots,x_N,\vp_N):\, (x_i,\vp_i,x_j,\vp_j)\in\Gamma_2\quad\forall (i,j)\right\} \,,
$$ 
with 
$$
  \Gamma_2 := \left\{ (x,\vp,x_*,\vp_*):\, \max\left\{ \left| \alpha\right|, \left| \alpha_*\right|\right\} > \frac{l}{2}\,,\quad \mbox{for }
  \alpha = \frac{(x_*-x)\cdot \omega_*^\bot}{\omega\cdot\omega_*^\bot} \,,\,
  \alpha_* = \frac{(x-x_*)\cdot \omega^\bot}{\omega_*\cdot\omega^\bot}\right\} \,,
$$ 
with $\omega = \omega(\vp)$, $\omega_* = \omega(\vp_*)$, $(a_1,a_2)^\bot = (-a_2,a_1)$. Note that $\alpha$ and 
$\alpha_*$ are determined such that $x+\alpha\omega = x_*+\alpha_*\omega_*$.

The collision rules are derived from the biological observations mentioned above. We assume that collisions between two bacteria 
$B$ and $B_*$ with pre-collisional states $(x,\vp)$ and, respectively, $(x_*,\vp_*)$ are instantaneous and can either lead to
	\begin{itemize}
		\item \textbf{Alignment}, if $\omega \cdot \omega_*>0$ (collision with pre-collisional angles less than $\pi/2$ apart), or to
		\item \textbf{Reversal} of both bacteria, if $\omega \cdot \omega_*<0$ (collision with pre-collisional angles greater than $\pi/2$ apart).
	\end{itemize}
Only \textit{binary collisions} are considered. As usual in kinetic theory, collisions between three or more bacteria at the same time are much less likely than binary collisions and are  therefore neglected.
By the same argument we neglect the limiting case $\omega\cdot\omega_* = 0$. 

For a precise formulation of the \textit{collision rules} we introduce the set of \textit{pre-collisional} states,
\begin{align*}
  \partial\Gamma_2^{out} := \Biggl\{ &(x,\vp,x_*,\vp_*)\in \partial\Gamma_2:\, \exists\,\alpha\in \left[-\frac{l}{2},\frac{l}{2}\right] :\, 
  x+\alpha\omega = x_* + \frac{l}{2}\omega_* \quad\mbox{or } \\
  & \exists\,\alpha_*\in \left[-\frac{l}{2},\frac{l}{2}\right] :\, x+\frac{l}{2}\omega = x_* + \alpha_*\omega_*\Biggr\}\,,
\end{align*}
and of \textit{post-collisional} states, $\partial\Gamma_2^{in} := \partial\Gamma_2 \setminus \partial\Gamma_2^{out}$, of a pair of bacteria.\smallskip

\noindent\textit{Alignment} between $(x,\vp)$ and $(x_*,\vp_*)$ happens, if $(x,\vp,x_*,\vp_*)\in \partial\Gamma_2^{out}$ and
   $$
     \vp_* \in \TAout := \{\psi\in\T:\, \omega(\vp)\cdot\omega(\psi)>0\} = \left(\vp-\frac{\pi}{2},\vp+\frac{\pi}{2}\right) \,.
   $$
   The alignment collision rule is (see Fig. \ref{fig:coll} (a)): \\
   $$(x,\vp),\, (x_*,\vp_*) \quad\rightarrow\quad (x',\vp'),\,(x',\vp') \qquad\mbox{with } x' = \frac{x+x_*}{2},\, \vp' = \frac{\vp+\vp_*}{2}.$$ \\

Note that the representation of $\TAout$ as an interval around $\vp$ is necessary for the above formula for the 
post-collisional angle $\vp'$ to provide the direction $\omega(\vp')$ lying between the pre-collisional directions $\omega(\vp)$ and 
$\omega(\vp_*)$.\smallskip

\noindent\textit{Reversal} between $(x,\vp)$ and $(x_*,\vp_*)$ happens, if $(x,\vp,x_*,\vp_*)\in \partial\Gamma_2^{out}$ and
   $$
     \vp_* \in \TRout := \{\psi\in\T:\, \omega(\vp)\cdot\omega(\psi)<0\}  \,.
   $$
   The reversal collision rule is (see Fig. \ref{fig:coll} (b)): \\		
   $$(x,\vp),\,(x_*,\vp_*) \quad\rightarrow\quad (x,\vp+\pi),\,(x_*,\vp_*+\pi).$$\\

\begin{figure}[h]
	\centering
	\begin{subfigure}{0.4\textwidth} 
		\includegraphics[width=\textwidth]{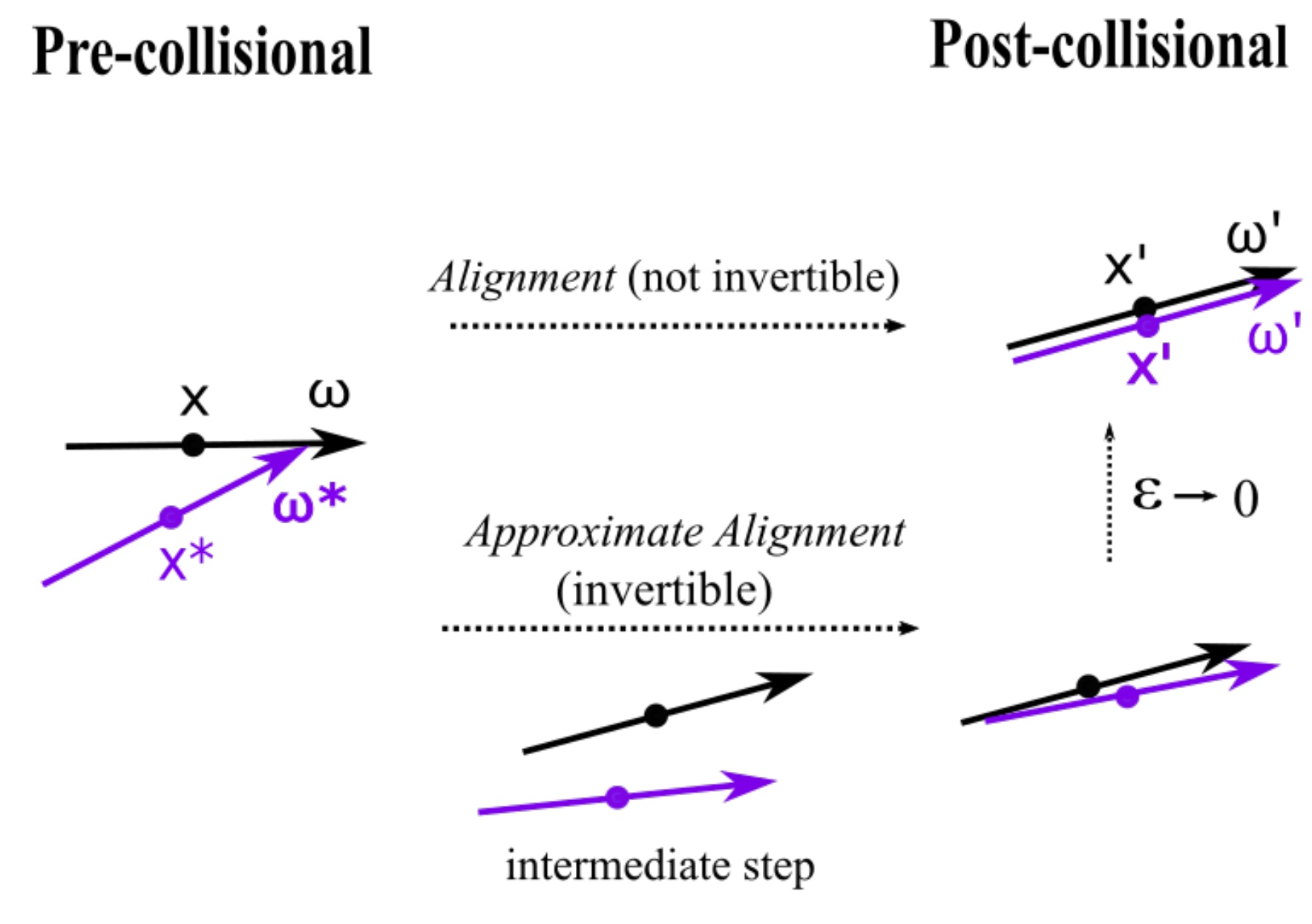}
		\caption{Alignment collisions} 
	\end{subfigure}
	\hspace{2em} 
	\begin{subfigure}{0.4\textwidth} 
		\includegraphics[width=\textwidth]{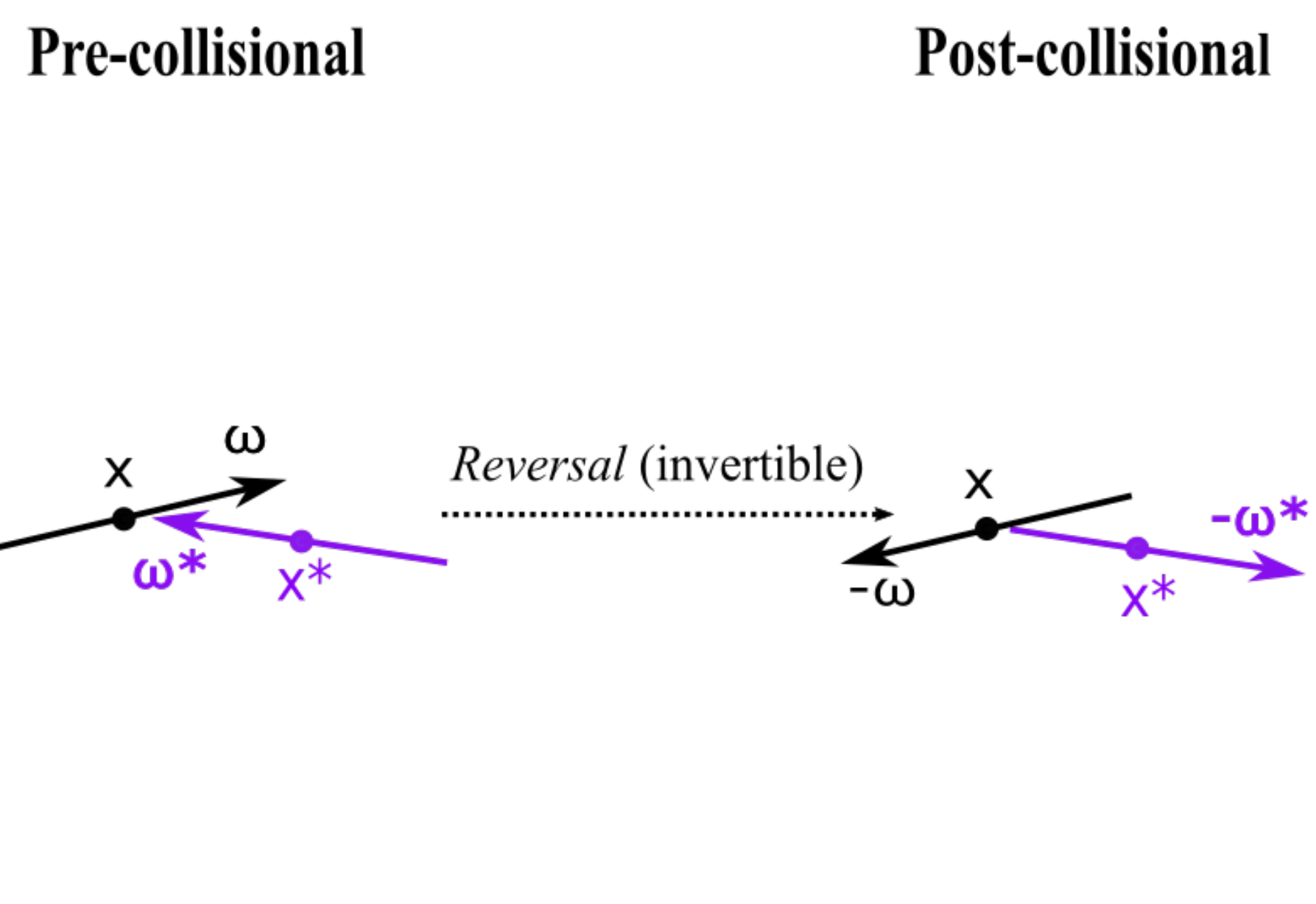}
		\caption{Reversal collisions} 
	\end{subfigure}
	\caption{Graphic illustration of the collision rules. (a): \textit{Alignment collisions} with two-step geometric algorithm to regularize it. (b): Already invertible \textit{reversal collisions}.}
	\label{fig:coll}
\end{figure}

\paragraph{Regularization of the alignment collisions:} In both types of collisions the pair of post-collisional states is in 
$\partial\Gamma_2^{in}$. After an alignment event its state space
velocity does, however, not point into the interior of $\Gamma_2$. Also the collision rule for alignment is obviously not invertible.
Since we intend to formulate a kinetic model following the standard derivation of the \textit{Boltzmann equation for hard spheres} \cite{boltzmann,Cercignani}, where the inverse of the collision rule is used, we shall introduce a regularization of the alignment
collisions, which will be removed again after the derivation.
The \textit{post-collisional} angles are reformulated such that the bacteria drift slightly apart after the collision: \smallskip
	
\noindent\textit{Regularized alignment:} Let $(x,\vp,x_*,\vp_*)\in \partial\Gamma_2^{out}$ with
$\vp_*\in \TAout$. With a small parameter $\ve>0$, the regularized collision rule is given by \quad
$(x,\vp),(x_*,\vp_*) \quad\rightarrow\quad (x',\vp'),(x_*',\vp_*')$, \\ with the rule
\begin{equation}\label{phi-AL}
\vp' = \frac{1-\ve}{2}\vp + \frac{1+\ve}{2}\vp_* \,,\qquad 
\vp_*' = \frac{1+\ve}{2}\vp + \frac{1-\ve}{2}\vp_* \,,
\end{equation}
for the angles.
The post-collisional centers are determined such that the post-collisional states are in $\partial\Gamma_2^{in}$, according to the 
following algorithm: First the bacteria are turned to the post-collisional directions around the pre-collisional centers, and then the 
centers are shifted towards each other, until the trailing end of one of them touches the other (see Fig. \ref{fig:coll} (a)). 


This leads to 
\begin{equation}\label{x-AL}
  x' = \frac{1+\ve A}{2} x + \frac{1-\ve A}{2} x_* \,,\qquad  x_*' = \frac{1-\ve A}{2} x + \frac{1+\ve A}{2} x_*\,,
\end{equation}
with $0<A = O(1)$ as $\ve\to 0$, depending on the pre-collisional state. There are two different versions for the formula for $A$ covering the cases where the pre-collisional leading end of $B$ is touching $B_*$ (and the post-collisional trailing end of $B_*$ is touching $B$) or vice versa.

The inversion of the collision rule is easy for the angles:
$$
  \vp = \frac{1+\ve}{2\ve}\vp_*' - \frac{1-\ve}{2\ve}\vp'\,,\qquad 
  \vp_* = \frac{1+\ve}{2\ve}\vp' - \frac{1-\ve}{2\ve}\vp_*'\,.
$$
For the cell centers it can be described by a geometric algorithm as above: First the bacteria are turned around their post-collisional
centers to the pre-collisional directions given above, leading to a forbidden state, where they cross each other. Then the centers are shifted apart until the leading end of one of them touches the other.

\paragraph{Probabilistic description:}
To derive the kinetic equation, we now reformulate the problem in terms of a probability density $P(\cdot,t)$ on $\Gamma_N$ at time $t\ge 0$. We assume \textit{indistinguishability} of the
bacteria, i.e. $P$ is invariant under permutations of the labels of the bacteria. It satisfies the \textit{Liouville equation}
	\begin{equation}\label{lio}
	\pa_t P+\sum_{i=1}^N v_i \cdot \nabla_{x_i} P = 0 \,,
	\end{equation}
where $v_i = s_0 \omega_i$, subject to boundary conditions, which are determined by the collision rules:
\begin{align}
	P(\dots, x',\vp', \dots, x_*', \vpa', \dots,t)&=F_{in}P(\dots, x, \vp, \dots, x_*, \vpa, \dots,t) \label{in1}\\
	&\text{for }\quad (x,\vp,x_*,\vp_*) \in \partial\Gamma_2^{out}, \quad  \omega\cdot\omega_*>0\,, \notag\\
	P(\dots, x,\vp^\downarrow, \dots, x_*, \vpa^\downarrow, \dots,t)&=P(\dots, x,\vp, \dots, x_*, \vpa, \dots,t) \label{in2}\\
	&\text{for}\quad (x,\vp,x_*,\vp_*) \in \partial\Gamma_2^{out}, \quad \omega\cdot\omega_*<0 \,, \notag
\end{align}
where $\vp^\downarrow = \vp+\pi$, $\vpa^\downarrow=\vpa+\pi$, and the relations between $(x,\vp,x_*,\vp_*)$ and 
$(x',\vp',x_*',\vp_*')$ in \eqref{in1}  are given by \eqref{phi-AL} and \eqref{x-AL}. The factor $F_{in}$ in \eqref{in1} is 
determined such that 
$$
  P(\dots, x',\vp', \dots, x_*', \vpa', \dots,t)|{v_*}'\cdot{\omega'}^\bot | d\sigma'
  = P(\dots, x, \vp, \dots, x_*, \vpa, \dots,t)|v_*\cdot\omega^\bot| d\sigma \,,
$$
where $d\sigma$ and $d\sigma'$ are the 5-dimensional surface measures on $\partial\Gamma_2^{out}$ and, respectively, 
$\partial\Gamma_2^{in}$. This guarantees particle conservation. No such factor is needed in \eqref{in2} since the reversal
collisions preserve the surface area as well as the normal component $|v_*\cdot\omega^\bot|$ of the flux.

We shall need a formula for $F_{in}$ for the situation, where the leading end of bacterium $B$ hits bacterium $B_*$ in
an alignment collision.
The corresponding part of $\partial\Gamma_2^{out}$ can be parametrized by $(x,\vp,\vp_*,\alpha)$ with
$$
  x_* = x + \frac{\ell}{2}\omega - \alpha\omega_* \,.
$$
Similarly, the parameters along the corresponding part of $\partial\Gamma_2^{in}$ can be taken as $(x',\vp',\vp_*',\alpha')$ 
with
$$
  x_*' = x' + \alpha'\omega' + \frac{\ell}{2}\omega_*' \,.
$$
A straightforward computation then gives 
$$
  F_{in} = \frac{|\omega_* \cdot \omega^{\perp}|}{\ve |\pa_{\alpha} \alpha'| |\omega_*' \cdot \omega'^{\perp}|} 
  \mathbb{1}_{\omega\cdot\omega_*>0}\,.
$$
Since $|\vp'-\vp_*'| = \ve |\vp-\vp_*|$, the inflow data vanish, whenever $\ve\pi/2<|\vp'-\vp_*'|<\pi/2$.

The $k$-bacteria marginals ($1\le k\le N$) of the distribution will be denoted by 
$$
    P_k(x_1,\vp_1,\ldots,x_k,\vp_k,t) := \int_{\Gamma_{N-k}^N(x_1,\vp_1,\ldots,x_k,\vp_k)} P(x_1,\vp_1,\ldots,x_N,\vp_N,t) 
    \prod_{j=k+1}^N dx_j d\vp_j \,,
$$
with 
$$
   \Gamma_{N-k}^N(x_1,\vp_1,\ldots,x_k,\vp_k) = \{(x_{k+1},\vp_{k+1},\ldots,x_N,\vp_N):\,
   (x_1,\vp_1,\ldots,x_N,\vp_N)\in\Gamma_N\}
$$
In order to obtain an evolution equation for the one-bacterium marginal $P_1(x,\vp,t)$, we integrate the Liouville equation (\ref{lio}) over $\Gamma_{N-1}^N(x,\vp)$, which gives
$$
  \pa_t P_1 + v \cdot \nabla_{x} P_1 
   + \sum_{j=2}^N   \int_{\Gamma_{N-1}^N(x,\vp)} v_j\cdot\nabla_{x_j} P \prod_{i=2}^N \; \md x_i \md \vp_i=0 \,.
$$
By the indistinguishability property, all the terms in the sum are identical, leading to
\begin{equation}\label{P1P2}
  \pa_t P_1 + v \cdot \nabla_{x} P_1 
   = - (N-1) \int_{\Gamma_1^2(x,\vp)} v_*\cdot \nabla_{x_*} P_2(x,\vp,x_*,\vp_*,t) dx_* d\vp_*  \,.
\end{equation}
An application of the divergence theorem gives an integration over 
$$
  (x_*,\vp_*)\in \partial\Gamma_1^2(x,\vp) \quad\Longleftrightarrow\quad (x,\vp,x_*,\vp_*) \in \partial\Gamma_2 \,,
$$
where the splitting 
\begin{eqnarray*}
  G-L &:=& (N-1)\int_{\T} \int_{-\ell/2}^{\ell/2} |v_*\cdot \omega^\bot| 
  P_2\left(x,\vp,x+\alpha\omega + \frac{\ell}{2}\omega_*,\vp_*\right) d\alpha\,d\vp_* \\
  && - (N-1)\int_{\T} \int_{-\ell/2}^{\ell/2} |v_*\cdot \omega^\bot| 
  P_2\left(x,\vp,x+\alpha\omega - \frac{\ell}{2}\omega_*,\vp_*\right) d\alpha\,d\vp_*
\end{eqnarray*}
of the right hand side of \eqref{P1P2}  into a \textit{gain term} and a \textit{loss term} corresponds to a
splitting into post-collisional states ($(x,\vp,x_*,\vp_*) \in \partial\Gamma_2^{in}$) and pre-collisional states 
($(x,\vp,x_*,\vp_*) \in \partial\Gamma_2^{out}$). Note that only those post-collisional states contribute, where the
trailing end of bacterium $B_*$ touches bacterium $B$, i.e. $x_*-\frac{\ell}{2}\omega_* = x+\alpha\omega$,
and only those pre-collisional states, where the leading end of $B_*$ touches $B$, i.e. 
$x_*+\frac{\ell}{2}\omega_* = x+\alpha\omega$.

The next step is to write the gain term in terms of pre-collisional states. In the part originating from reversal collisions
it is straightforward to use the boundary conditions \eqref{in2} to obtain
$$
   G_{REV}(x,\vp) = (N-1)s_0\int_{\TRout} \int_{-\ell/2}^{\ell/2} |\omega_*\cdot \omega^\bot| 
  P_2\left(x,\vp^\downarrow,x+\alpha\omega - \frac{\ell}{2}\omega_*,\vpa^\downarrow\right) d\alpha\,d\vp_* \,,
$$
where also the coordinate change $\alpha\to-\alpha$ has been carried out. For the alignment collisions a little more care
is necessary. For easier use of our earlier notation we write
\begin{eqnarray*}
   G_{AL,\ve}(x',\vp') &=& (N-1)s_0\int_{\TAoutp} \int_{-\ell/2}^{\ell/2} |\omega_*'\cdot {\omega'}^\bot| 
  P_2\left(x',\vp',x'+\alpha'\omega' + \frac{\ell}{2}\omega_*',\vp_*'\right) d\alpha'\,d\vp_*' \\
  &=& (N-1)s_0\int_{\TAoutp} \int_{-\ell/2}^{\ell/2} |\omega_*'\cdot {\omega'}^\bot| F_{in}
  P_2\left(x,\vp,x+\frac{\ell}{2}\omega - \alpha\omega_*,\vp_*\right) d\alpha'\,d\vp_*' \\
  &=& \frac{2(N-1)s_0}{1-\ve}\int_{\TAinp} \int_{-\ell/2}^{\ell/2} |\omega_*\cdot \omega^\bot|
  P_2\left(x,\vp,x+\frac{\ell}{2}\omega - \alpha\omega_*,\vp_*\right) d\alpha\,d\vp_* \,,
\end{eqnarray*}
where in the last line $\TAinp = \{\vp_*:\, |\vp'-\vp_*|\le (1-\ve)\pi/4\}$ denotes the set of all angles $\vp_*$ which, after an alignment collision with collision partner
$$
  \vp = \frac{2\vp' - (1+\ve)\vp_*}{1-\ve}
$$
(as a consequence of \eqref{phi-AL}) produce the post-collisional angle $\vp'$. Also $x$ can be expressed in terms of $x'$, $\vp'$, $\alpha$, and $\vp_*$, satisfying $x=x'+O(\ell)$ for small $\ell$. Note that this representation is sufficient for the Boltzmann-Grad limit, which we will perform next and where $\ell$ is assumed to be small compared to a reference length. The computations have involved the use of the boundary conditions \eqref{in1}
and the coordinate change $(\alpha',\vp_*') \to (\alpha,\vp_*)$, according to the rules for the regularized alignment collisions.

\paragraph{Scaling and Boltzmann-Grad limit: }
We choose as macroscopic length scale the total length of $N-1$ bacteria, $\mathcal{L}=(N-1)\ell$, and introduce the nondimensionalization
$$
   x \to \mathcal{L} x \,,\quad t \to \frac{\mathcal{L}}{s_0} t \,,\quad P_k \to \mathcal{L}^{-2k} P_k \,,\quad \alpha\to \ell\alpha \,,
$$
leading to the dimensionless version of the equation for the one-bacterium marginal:
$$
   \pa_t P_1 + \omega \cdot \nabla_x P_1 = G_{AL,\ve} + G_{REV} - L \,,
$$
where
\begin{eqnarray*}
  L(x,v) &=& \int_{\T} \int_{-1/2}^{1/2} |\omega_*\cdot \omega^\bot| 
  P_2\bigl(x,\vp,x+\delta(\alpha\omega - \omega_*/2),\vp_*\bigr) d\alpha\,d\vp_*\,,\\
  G_{REV}(x,\vp) &=& \int_{\TRout} \int_{-1/2}^{1/2} |\omega_*\cdot \omega^\bot| 
  P_2\left(x,\vp^\downarrow,x+\delta(\alpha\omega - \omega_*/2),\vpa^\downarrow\right) d\alpha\,d\vp_* \,,\\
  G_{AL,\ve}(x,\vp) &=&  \frac{2}{1-\ve}\int_{\TAin} \int_{-1/2}^{1/2} |\omega_*\cdot \tilde\omega^\bot|
  P_2\bigl(\tilde x,\tilde\vp,\tilde x+\delta(\tilde\omega/2 - \alpha\omega_*),\vp_*\bigr) d\alpha\,d\vp_* \,,
\end{eqnarray*}
with 
$$
  \delta = \frac{\ell}{\mathcal{L}} \,,\qquad   \tilde\vp = \frac{2\vp - (1+\ve)\vp_*}{1-\ve} \,,\qquad 
 \tilde x = x + O(\delta) \quad\mbox{as } \delta\to 0 \,.
$$
The Boltzmann-Grad limit is the large particle number limit $N\to\infty$, i.e. $\delta\to 0$. As usual,
the \textit{molecular chaos} assumption \cite{Cercignani} will be used. Roughly speaking, it amounts to assuming that 
initially the probability distributions of the bacteria are pairwise independent and that any pair of bacteria collides at most 
once, such that the independence is still valid for each pre-collisional state. This is the reason for writing the collision
integrals in terms of pre-collisional states. As a consequence, assuming $P_1 \to f$ implies $P_2 \to f \otimes f$ as $N \to \infty$, 
wherever it occurs in the equation. In the limit, we obtain the Boltzmann-type equation
\begin{equation}\label{Boltzmann}
\begin{split}
	&\pa_t f + \omega \cdot \nabla_x f =  G_{AL,\ve}(f,f) + G_{REV}(f,f) - L(f,f) \\
	 & = \frac{2}{1-\ve}\int_{\TAin}  b(\tilde\vp,\vpa) f(x,\tilde\vp)f(x,\vp_*) d\vpa  
	+ \int_{\TRout} b(\vp,\vpa) f(x, \vp^\downarrow)f(x, \vpa^\downarrow) d\vpa \\
	& - \int_{\T} b(\vp,\vpa) f(x, \vp)f(x, \vpa) d\vpa \,,
\end{split}
\end{equation}
with $b(\vp,\vpa) = |\omega_*\cdot \omega^\bot| = |\sin(\vp-\vpa)|$.
The collision integrals are now written as bilinear operators where, abusing notation, we have kept the same names.
	
\paragraph{Alignment limit:}
It is now straightforward to remove the regularization of the alignment collisions, i.e. to carry out the limit $\ve\to 0$,
leading to our \textit{kinetic model for myxobacteria:}
\begin{equation}\label{myxo}
\begin{split}
	&\pa_t f + \omega \cdot \nabla_x f = Q(f,f) := G_{AL}(f,f)+G_{REV}(f,f)-L(f,f) \\
	&= 2\int_{\TAin}  b(\tilde\vp,\vpa) \tilde f f_* d\vpa  
	+ \int_{\TRout} b(\vp,\vpa) f^\downarrow f_*^\downarrow d\vpa 
         - \int_{\T} b(\vp,\vpa) f f_* d\vpa \,,
\end{split}
\end{equation}
with
$$
  \TRout = \left(\vp+\frac{\pi}{2},\vp+\frac{3\pi}{2}\right) \,,\qquad \TAin = \left( \vp-\frac{\pi}{4}, \vp+\frac{\pi}{4}\right) \,,
$$ $$
  \tilde\vp = 2\vp-\vpa \,,\qquad \vp^\downarrow = \vp+\pi \,,\qquad \vpa^\downarrow = \vpa+\pi\,.
$$
Note that $\vpa\in\TAin$ satisfies $\omega(\tilde\vp)\cdot\omega(\vpa)>0$ and that $\TRout$ is a representation of the set 
of all $\vpa\in\T$ satisfying $\omega(\vp)\cdot\omega(\vpa)<0$.

\paragraph{'Maxwellian myxos':}
The factor $b(\vp,\vpa) = |\omega_*\cdot\omega^\bot| = |\sin(\vp-\vpa)|$ in the collision integrals is a consequence of 
the rod shape
of the bacteria. It gives the rate of collisions between bacteria with the directions $\vp$ and $\vpa$. Assuming instead
bacteria with circular shape makes the collision rate independent from the movement direction. By analogy to a similar 
simplification of the gas dynamics Boltzmann equation \cite{Cercignani}, we use the name \textit{Maxwellian myxos} for 
this imagined species, modeled by \eqref{myxo} with $b(\vp,\vpa)\equiv 1$.

\section{Properties of the collision operator}\label{sec:3}
	
\paragraph{Collision invariants and conservation laws:}
In the following it will be convenient to also split the loss term of the collision operator into alignment and reversal parts:
\begin{eqnarray}
   Q(f,f) &=& Q_{AL}(f,f) + Q_{REV}(f,f) \nonumber\\
    &=& \int_{\T} \left( 2b(\tilde\vp,\vpa) \mathbb{1}_{\vpa\in\TAin} \tilde f f_*  
          - b(\vp,\vpa) \mathbb{1}_{\vpa\in\TAout} f f_*\right)d\vpa  \label{Q:AL+REV}\\
    && + \int_{\TRout} b(\vp,\vpa) (f^\downarrow f_*^\downarrow - f f_*) d\vpa \,,\nonumber
 \end{eqnarray}
A weak formulation of the alignment operator is obtained by integration against a test function $\psi(\vp)$, the coordinate
change $\tilde\vp = 2\vp-\vpa \to \vp$, and subsequent symmetrization:
\begin{equation}\label{AL-weak}
  \int_{\T} Q_{AL}(f,f)\psi\,d\vp = \int_{\T}\int_{\TAout} b(\vp,\vpa) f f_* \left(\psi\left(\frac{\vp+\vpa}{2}\right)
  - \frac{\psi(\vp)+\psi(\vpa)}{2}\right) d\vpa\,d\vp
\end{equation}
This shows that the space of collision invariants of $Q_{AL}$ is two-dimensional and spanned by $\psi=1$ and $\psi=\vp$.
Furthermore, with $\psi = \vp^2$, we obtain
$$
  \int_{\T} Q_{AL}(f,f)\vp^2\,d\vp = -\frac{1}{4}\int_{\T}\int_{\TAout} b(\vp,\vpa) f f_* (\vp-\vpa)^2 d\vpa\,d\vp \le 0 \,.
$$
Therefore $Q_{AL}(f,f)=0$ implies that for $f(\vp)\ne 0$, $f(\vpa)$ vanishes for all $\vp\ne\vpa\in\TAout$. As a consequence, equilibria are concentrated at isolated angles with a pairwise distance bigger than $\pi/2$, implying that there are at
most three of them. Thus, every equilibrium distribution $f$ of $Q_{AL}$ can be written as
$$
  f(\vp) = \rho_1 \delta(\vp-\vp_1) + \rho_2 \delta(\vp-\vp_2) +\rho_3 \delta(\vp-\vp_3) \,,
$$
with $\rho_j\ge 0$ and dist$_{\T}(\vp_i,\vp_j)> \pi/2$, $i\ne j$, where
$$
  {\rm dist}_{\T}(\vp,\vpa) := \min_{k\in\mathbb{Z}} |\vp - \vpa + 2k\pi| \le \pi\,.
$$

The weak formulation of the reversal operator can be written as
\begin{eqnarray}
  \int_{\T} Q_{REV}(f,f)\psi\,d\vp &=& \frac{1}{2}\int_{\T}\int_{\TRout} b(\vp,\vpa) f f_* 
    \left(\psi^\downarrow + \psi_*^\downarrow - \psi - \psi_*\right) d\vpa\,d\vp \nonumber\\
  &=& \frac{1}{2}\int_{\T}\int_{\TAout} b(\vp,\vpa) f f_*^\downarrow
    \left(\psi^\downarrow + \psi_* - \psi - \psi_*^\downarrow\right) d\vpa\,d\vp \,,\label{REV-weak}
\end{eqnarray}
where the first equality is obtained analogously to the treatment of the alignment operator, and the second equality
is due to the coordinate change $\vpa \leftrightarrow \vpa^\downarrow$. Both forms
show that all $\pi$-periodic functions are collision invariants. However, the second representation reveals the additional
collision invariant $\psi(\vp)=\vp$. It is obvious that all $\pi$-periodic functions are
equilibria of $Q_{REV}$. However, the set of equilibria is larger: Let $g:\, (\pi/4,3\pi/4)\to\mathbb{R}_+$ be arbitrary,
 $\lambda \ge 0$, and let
 \begin{equation}\label{REV-equ}
    f(\vp) := \left\{ \begin{array}{ll} g(\vp) & \mbox{for } \vp\in\T_+ := (\pi/4,3\pi/4) \,,\\
                                                    \lambda g(\vp+\pi) & \mbox{for } \vp\in\T_- := (-3\pi/4,-\pi/4) \,,\\
                                                    0 & \mbox{else.}  \end{array}\right.
 \end{equation}
 Then it is easily checked that $Q_{REV}(f,f)=0$. We see that the set of functions unaffected by reversal collisions contains functions describing bacteria colonies which can be separated into two groups, one moving upwards with direction $\vp \in \T_+$, the other downwards with $\vp \in \T_-$ and whose distribution in each group is equal up to a proportionality constant $\lambda \geq 0$. It is important to note that the boundary angles of $\T_+$ and $\T_-$ are $\pi/2$ apart, so that reversal collisions can only occur between two individuals from different groups.\\
 The question of a characterization of the whole set of equilibria of $Q_{REV}$ seems rather difficult and is left open.

Since the collision invariants of $Q_{AL}$ are also collision invariants of $Q_{REV}$, solutions of \eqref{myxo} 
satisfy two conservation laws, \textit{conservation of the number of bacteria,} 
	\begin{align}\label{massconservation}
	\pa_t \rho + \nabla_x \cdot (\rho u) = 0 \,,
\end{align}
with the usual definition of \textit{number density} and \textit{flux:}
$$
    \rho(x,t):= \int_{\T} f(x, \vp, t)  d\vp \,,\qquad  \rho u(x,t) := \int_{\T} \omega(\vp) f(x,\vp,t) d\vp\,,
$$
and 
\begin{equation}\label{phi-cons}
   \partial_t \int_{\T} \vp f\,d\vp + \nabla_x\cdot \int_{\T} \vp\, \omega f\,d\vp = 0 \,.
\end{equation}
Note that this second conservation law depends on the representation of $\T$. However, the differences are only up
to adding a multiple of the bacteria number.

The only equilibria of $Q_{AL}$, which are also equilibria of $Q_{REV}$, are of the form
\begin{equation}\label{Equ}
   f_\infty(\vp) = \rho_+ \delta(\vp-\vp_+) + \rho_- \delta(\vp-\vp_+^\downarrow) \,,
\end{equation}
with arbitrary $\rho_\pm \ge 0$, $\vp_+\in\T$. This raises the problem that there are three free parameters, $\rho_+$, $\rho_-$, $\vp_+$, in the 
equilibrium distribution as opposed to only two conservation laws (\ref{massconservation}) and (\ref{phi-cons}).
		
\paragraph{Two group initial data:}	
The form \eqref{REV-equ} of reversal equilibria suggests to consider initial conditions 
$$
     f(x,\vp,0) = f_I(x,\vp) \,,
$$ 
satisfying
\begin{align}\label{spin}
	\supp(f_I(x,.)) \subset \T_+ \cup \T_-,  \qquad \forall x \in \R^2 \,,
\end{align}
\begin{figure}[h]
	\centering
	\includegraphics[width=7cm]{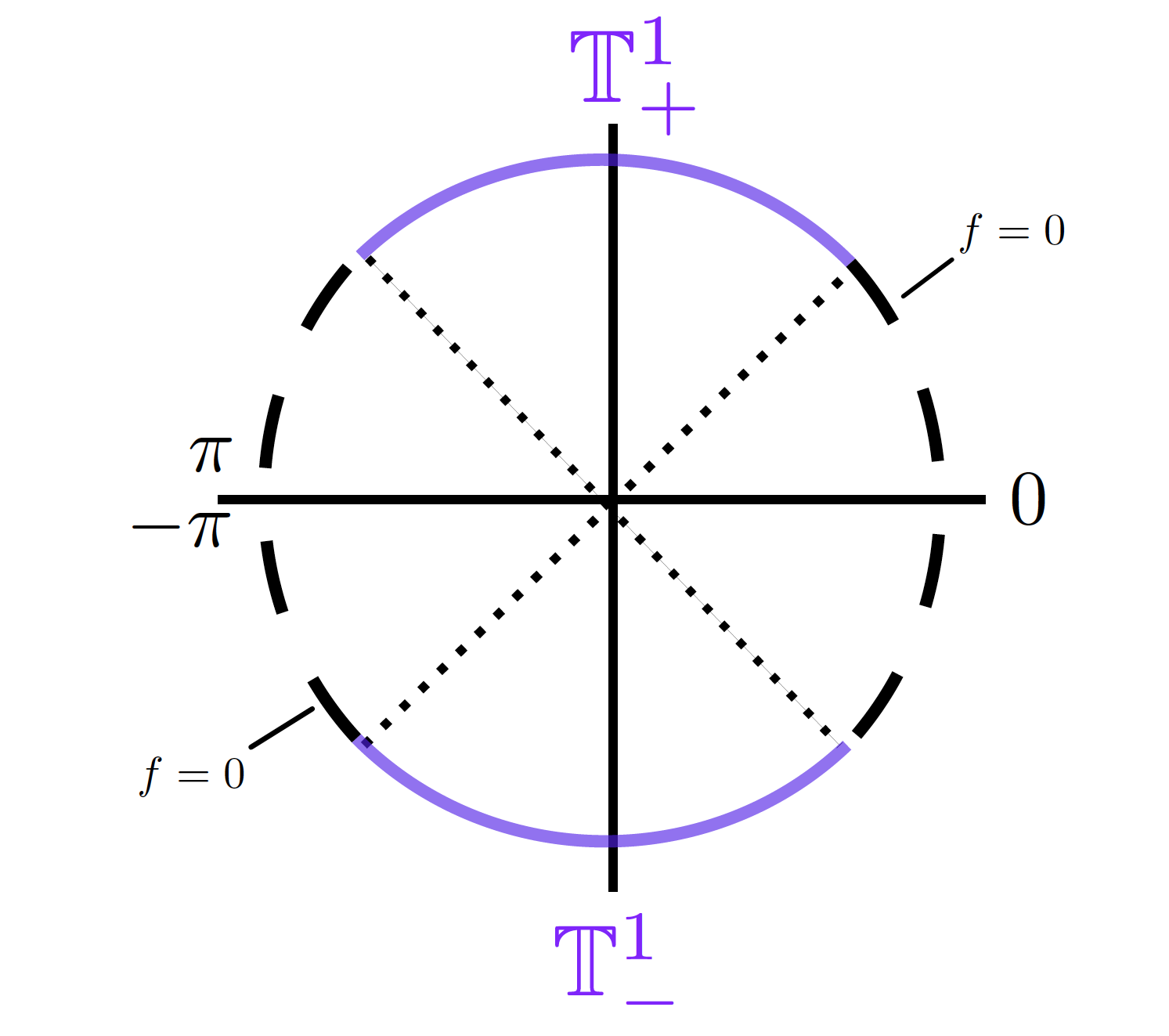}
	\caption{Support of two group data (solid lines, purple).}\label{pm}
\end{figure}
i.e., with angles in the opposite groups $\T_+$ and $\T_-$ (see Fig. \ref{pm}). It is easily seen that the property \eqref{spin} 
is propagated by \eqref{myxo}. Indeed, \textit{alignment collisions} are only possible between two individuals from the same group, producing post-collisional angles in the same group. \textit{Reversal interactions} can only occur between bacteria from different groups, causing the two individuals to swap groups. These observations imply that in this special situation the bacteria numbers in each group are conserved. With the notation 
$$
  \rho_\pm = \int_{\T_\pm} f\,d\vp \,,\qquad \rho_\pm u_\pm = \int_{\T_\pm} \omega f\,d\vp \,,
$$
we have
$$
   \partial\rho_\pm + \nabla_x\cdot (\rho_\pm u_\pm) = 0 \,,
$$
where the sum gives \eqref{massconservation}, of course. Thus, there is one additional conservation law, bringing the total
number up to three, the dimension of the set of equilibria. This will allow us to perform the (formal) macroscopic limit in Section 6.
	
\section{The spatially homogeneous problem}
	 
\paragraph{Existence and uniqueness of solutions:} We consider the initial value problem 	
\begin{align}\label{hom}
	&\pa_t f = Q(f,f), \qquad \mbox{in } \T\times(0,\infty)\\
	&f(\vp,0) = f_I(\vp) \,, \qquad \vp \in \T, \notag
\end{align}
with the collision operator as in \eqref{myxo} and no restriction on $f_I$. Existence and uniqueness in $\lone$ will be shown by the Picard theorem
since, by the boundedness of the collision cross-section $b$, the collision operator can be shown to be Lipschitz
continuous.
	
\begin{theorem}
Let $b\in L^\infty(\T\times\T)$ and $f_I \in L^1_+(\T)$. Then \eqref{hom}
has a unique global solution $f \in C\left([0,\infty),L^1_+(\T) \right) $.
\end{theorem}
	
\begin{proof}
Let $f,g \in \lone$ with $\|f\|_{\lone}, \|g\|_{\lone}\le \rho$, and $\bar b := \|b\|_{L^\infty(\T\times\T)}$. We split the collision operator as in \eqref{myxo}:
$$
  \|G_{AL}(f,f) - G_{AL}(g,g)\|_{\lone} \le 2\bar b\int_{\T}\int_{\TAin}  |\tilde f f_* - \tilde g g_*| d\vpa d\vp
  =  \bar b\int_{\T}\int_{\TAout}  |f f_* - g g_*| d\vpa d\vp\,,
$$ 
with the change of variables $\tilde\vp\to\vp$ as in the previous section. Further estimation gives
\begin{eqnarray*}
  \|G_{AL}(f,f) - G_{AL}(g,g)\|_{\lone} &\le& \bar b\int_{\T}\int_{\TAout} \left( |f|\, |f_* - g_*|  + |f-g|\, |g_*| \right)d\vpa d\vp \\
  &\le& 2\bar b \rho\, \|f-g\|_{\lone}
\end{eqnarray*}
For the reversal term we have
\begin{eqnarray*}
  \|G_{REV}(f,f) - G_{REV}(g,g)\|_{\lone} &\le& \bar b\int_{\T}\int_{\TRout} |f^\downarrow f_*^\downarrow 
    - g^\downarrow g_*^\downarrow| d\vpa d\vp \\
    &=& \bar b\int_{\T}\int_{\TRout} |f f_* - g g_*| d\vpa d\vp \le 2\bar b \rho\, \|f-g\|_{\lone} \,,
\end{eqnarray*}
with $(\vp^\downarrow,\vpa^\downarrow) \to (\vp,\vpa)$. An analogous estimate for the loss term finally gives
$$
  \|Q(f,f)-Q(g,g)\|_{\lone} \le 6\bar b \rho \|f-g\|_{\lone} \,.
$$
Therefore a unique local solution exists by Picard iteration. Nonnegativity and conservation of the number of bacteria,
i.e. of the $\lone$-norm, are obvious, the latter implying global existence. 
\end{proof}

\paragraph{Convergence to equilibrium:}
We study the convergence of solutions of the spatially homogeneous problem \eqref{hom} to equilibria of the 
form \eqref{Equ} as $t\to\infty$. We have, however, only partial results in this direction due to two difficulties. The first one 
is the lack of a third conservation law for general initial data. We shall therefore restrict our attention to 
\textit{two-group initial data} $f_I$ satisfying (\ref{spin}). In this case the conservation of
$$
    \int_{\T_+} f\,d\vp \,,\qquad \int_{\T_-} f\,d\vp \,,\qquad  \mbox{and}\quad \int_{\T} \vp f\,d\vp \,,
$$
allows to determine the parameters in \eqref{Equ} from the initial data:
$$
   \rho_+ = \int_{\T_+} f_I\,d\vp \,,\qquad \rho_- = \int_{\T_-} f_I\,d\vp \,,\qquad  
   \vp_+ =  \frac{1}{\rho_+ + \rho_-} \left(\int_{\T_+} \vp f_I\,d\vp + \int_{\T_-}\vp^\downarrow f_I\,d\vp\right)\,.
$$
Note that $\vp_+\in\T_+$ is an average angle where, however, angles in $\T_-$ are mapped to $\T_+$ by reversal.

First, we state a preliminary result on the decay of the variance 
$$
   V[f] := \int_{\T_+} (\vp - \vp_+)^2 f \,d\vp
$$
for the even more restricted case of \emph{one-group initial data} supported in $\T_+ = (\pi/4,3\pi/4)$, 
where only alignment collisions occur.

\begin{lemma}\label{lem:AL}
Let $f_I \in L_+^1(\T)$ with supp$(f_I) \subset \T_+$ and let $f$ be a solution of \eqref{hom}. Then \\
a) for Maxwellian myxos, i.e. $b(\vp,\vpa)\equiv 1$,
$$
   V[f(\cdot,t)] = e^{-t \rho_+/2} V[f_I] \,,
$$
b) and for rod shaped myxobacteria, i.e. $b(\vp,\vpa)=|\sin(\vp-\vpa)|$,
$$
   \frac{1}{\rho_+} \left( M_{1,I}^{-1} + 2t\right)^{-2} \le V[f(\cdot,t)] \le \left( V[f_I]^{-1/2} + \kappa t \right)^{-2} \,,
$$
with 
$$ 
  M_{1,I} := \int_{\T_+} |\vp - \vp_+| f_I \,d\vp \,,\qquad    \kappa = \frac{\sqrt{\rho_+}}{4\pi}\,.
$$
\end{lemma}

\begin{rem}
The result of Lemma \ref{lem:AL} b) corresponds to \textit{Haff's law} \cite{haff} for the spatially homogeneous dissipative 
Boltzmann equation, stating that the variance of the distribution decays like $t^{-2}$. There the degeneracy of the collision
cross section is the same as here. Our proof follows along the lines of \cite{alonso2}.
\end{rem}

\begin{proof} 
For the computation of the time derivative of the variance along solutions of \eqref{hom} the formula \eqref{AL-weak}
 with $\psi(\vp) = \mathbb{1}_{\T_+}(\vp)(\vp-\vp_+)^2$ can be used:
\begin{equation}\label{HAL-dissipation}
   \frac{dV[f]}{dt} = -\frac{1}{4} \int_{\T_+} \int_{\T_+} b(\vp,\vpa) f f_*(\vp-\vpa)^2 d\vpa d\vp \,.
\end{equation}
a) We compute
\begin{eqnarray*}
   \frac{dV[f]}{dt} = -\frac{1}{4}\int_{\T_+} \int_{\T_+} f f_*(\vp-\vp_+ + \vp_+ - \vpa)^2 d\vpa d\vp 
   = -\frac{\rho_+}{2} \int_{\T_+} f (\vp-\vp_+)^2 d\vp = -\frac{\rho_+}{2} V[f] \,.
\end{eqnarray*}
b) Since $|\vp-\vpa|\le \pi/2$ in the right hand side of \eqref{HAL-dissipation}, we have 
$$
  b(\vp,\vpa) = |\sin(\vp-\vpa)| \ge \frac{2}{\pi}|\vp-\vpa| \,,
$$
and therefore, using the Jensen inequality twice,
\begin{eqnarray*}
  &&\int_{\T_+} \int_{\T_+} b(\vp,\vpa) f f_*(\vp-\vpa)^2 d\vpa d\vp \ge
  \frac{2}{\pi} \int_{\T_+} f \left(\int_{\T_+} f_* |\vp-\vpa|^3 d\vpa \right)d\vp \\
  && \ge \frac{2\rho_+}{\pi} \int_{\T_+} f \left|\int_{\T_+} \frac{f_*}{\rho_+} (\vp-\vpa) d\vpa \right|^3 d\vp
  = \frac{2\rho_+}{\pi} \int_{\T_+} f \left| \vp-\vp_+\right|^3 d\vp \\
  && \ge \frac{2\rho_+^2}{\pi} \left(\int_{\T_+} \frac{f}{\rho_+}  \left(\vp-\vp_+\right)^2 d\vp \right)^{3/2}
  = \frac{2\sqrt{\rho_+}}{\pi} \left(V[f] \right)^{3/2} \,,
\end{eqnarray*}
giving
\begin{eqnarray*}
  \frac{dV[f]}{dt} &\le& - \frac{\sqrt{\rho_+}}{2\pi} V[f]^{3/2} \,,
\end{eqnarray*}
implying the upper bound by solving the corresponding differential equation. A lower bound is first derived for 
$$
   M_1[f](t) := \int_{\T_+} |\vp - \vp_+| f \,d\vp \,.
$$
We again use \eqref{AL-weak}, now with $\psi(\vp) = \mathbb{1}_{\T_+}(\vp)|\vp-\vp_+|$:
$$
   \frac{dM_1[f]}{dt} = -\frac{1}{2} \int_{\T_+} \int_{\T_+} |\sin(\vp-\vpa)| f f_*\left(|\vp-\vp_+| + |\vpa-\vp_+| - |\vp+\vpa-2\vp_+|  \right) d\vpa d\vp \,.
$$
With the elementary inequalities (see also \cite[equ. (3.3)]{alonso2} for the second)
\begin{eqnarray*}
  (|a|+|b|-|a+b|) |\sin(a-b)| \le (|a|+|b|-|a+b|) |a-b|\le 4|a|\,|b| \,,
\end{eqnarray*}
we obtain
$$
   \frac{dM_1[f]}{dt} \ge -2M_1[f]^2 \,,
$$
implying
$$
   M_1[f](t) \ge \left(M_{1,I}^{-1} + 2t\right)^{-1}  \,.
$$
An application of the Cauchy-Schwarz inequality $M_1[f]^2 \le \rho_+ V[f]$ concludes the proof.
\end{proof}

Lemma \ref{lem:AL} can be interpreted as a convergence result with respect to the \textit{Wasserstein distance} \cite{Villani}. In particular, for
$f,g\in\mathcal{P}(\T)$ (the space of probability measures), the Wasserstein distance with quadratic cost is defined by
$$
   W_2^{\T}(f,g):= \inf_{\pi \in \Pi(f,g)}{\left(\iint_{\T \times \T} {\rm dist}_{\T}(\vp_1,\vp_2)^2 \; d\pi(\vp_1,\vp_2)\right)} ^{1/2},
$$ 
where $\Pi(f,g)\subset \mathcal{P}(\T\times\T)$ is the set of all transference plans $\pi$, satisfying 
$\pi(\cdot,\T)=f$, $\pi(\T,\cdot)=g$. We shall also use the straightforward extension of the definition to pairs of measures
with the same total mass, not necessarily equal to one. It is well known that for $g(\vp) = m\delta(\vp-\hat\vp)$ the only 
possible transference plan is $\pi = (f\otimes g)/m$ and therefore
\begin{equation}\label{W2=variance}
   W_2^{\T}(f,g)^2 = \int_{\T} {\rm dist}_{\T}(\vp,\hat\vp)^2 f \,d\vp \,,
\end{equation}
implying for distributions with support $\T_+$ as in Lemma \ref{lem:AL} that $V[f] = W_2^{\T}(f,f_\infty)^2$.

Since for the two-group case we are dealing with distributions, which are the sums of two point masses, we shall need the following result.

\begin{lemma}\label{lem:W2split}
Let $f,g\in\mathcal{P}(\T)$, ${\rm supp}(f),\,{\rm supp}(g) \subset \T_+\cup\T_-$, $f(\T_\pm)=g(\T_\pm)$. Then
$$
   W_2^{\T}(f,g)^2 = W_2^{\T_+}\left(f,g\right)^2 + W_2^{\T_-}\left(f,g\right)^2 \,,
$$
where on the right hand side $f,g$ denote the restrictions to $\T_+$ and, respectively, $\T_-$.
\end{lemma}

\begin{rem}
The result is actually a rather obvious consequence of the fact that the distance between points within $\T_\pm$ is 
never larger than the distance between a point in $\T_+$ and a point in $\T_-$, with the consequence that there always 
exists an optimal transference plan transferring only within the two groups.
\end{rem}

\begin{proof}
We only give a proof for the case where $f$ and $g$ are sums of point measures, since the result then follows by a density argument.
So let
$$
    f = \sum_{i=1}^M f_i \delta_{\vp_i} \,,\qquad g = \sum_{j=1}^N g_j \delta_{\psi_j} \,.
$$ 
A transference plan is then determined by a matrix $\pi\in \R^{M\times N}$ with nonnegative entries, such that
$$
   \sum_{j=1}^N \pi_{ij} = f_i \,,\qquad \sum_{i=1}^M \pi_{ij} = g_j \,.
$$
The statement of the lemma means that there exists an optimal $\pi$ such that 
\begin{equation}\label{good-transference}
   \pi_{ij} > 0  \qquad\Longleftrightarrow\qquad \vp_i,\psi_j \in \T_+ \quad\mbox{or}\quad  \vp_i,\psi_j \in \T_-\,.
\end{equation}
Let now $\pi$ be a general transference plan and assume that there exists $(i,j)$ such that $\vp_i\in\T_+$, $\psi_j\in\T_-$,
$\pi_{ij}>0$. This means that some mass is transferred from $\T_+$ to $\T_-$. Since the total masses are the same in
both groups, the mass balance requires that also some mass is transferred from $\T_-$ to $\T_+$, i.e. there exists
$(i_*,j_*)$ such that $\vp_{i_*}\in\T_-$, $\psi_{j_*}\in\T_+$, $\pi_{i_* j_*}>0$.

The idea is that in this situation the transference plan can be improved by moving mass $m:= \min\{\pi_{ij},\pi_{i_* j_*}\}$
in a cheaper way by the changes 
$$
  \pi_{ij} \to \pi_{ij} - m \,,\quad \pi_{i_* j_*} \to \pi_{i_* j_*} - m \,,\quad 
  \pi_{ij_*} \to \pi_{ij_*} + m \,,\quad \pi_{i_* j} \to \pi_{i_* j} + m \,.
$$
This means that the contribution 
$$
   m\left({\rm dist}_{\T}(\vp_i,\psi_j)^2 + {\rm dist}_{\T}(\vp_{i_*},\psi_{j_*})^2\right) \ge \frac{m\pi^2}{2}
$$
to the total cost is replaced by 
$$
    m\left((\vp_i - \psi_{j_*})^2 + (\vp_{i_*} - \psi_j)^2\right) \le \frac{m\pi^2}{2} \,,
$$
and in the improved transference plan either $\pi_{ij}$ or $\pi_{i_* j_*}$ is replaced by zero. Iterating the procedure, 
an improved transference plan satisfying \eqref{good-transference} is reached in finitely many steps.
\end{proof}

For the two-group case with equilibrium $f_\infty$ given in \eqref{Equ}, it seems natural to examine the evolution of
the Wasserstein distance 
$$
  W_2^{\T}(f,f_\infty)^2 = \int_{\T_+} (\vp - \vp_+)^2 f \,d\vp + \int_{\T_-} (\vp - \vp_+ + \pi)^2 f \,d\vp\,.
$$
For the computation of its time derivative along solutions of \eqref{hom} the formulas \eqref{AL-weak}, \eqref{REV-weak}
with $\psi(\vp) = \mathbb{1}_{\T_+}(\vp)(\vp-\bar\vp_+)^2 + \mathbb{1}_{\T_-}(\vp)(\vp-\bar\vp_-)^2$ can be used:
\begin{eqnarray}
   \frac{d}{dt}W_2^{\T}(f,f_\infty)^2 &=& -\frac{1}{4} \int_{\T_+} \int_{\T_+} b(\vp,\vpa) f f_*(\vp-\vpa)^2 d\vpa d\vp \nonumber\\
   &&-\frac{1}{4} \int_{\T_-}\int_{\T_-} b(\vp,\vpa) f f_*(\vp-\vpa)^2 d\vpa d\vp  \le 0\,.\label{HAL-dissipation2}
\end{eqnarray}
Note that the reversal collisions do not contribute to the right hand side which vanishes, whenever concentration is
reached in both groups, even when the two concentration angles are not opposite each other. Therefore it is not possible
to derive a differential inequality for $W_2^{\T}(f,f_\infty)$ as in the proof of Lemma \ref{lem:AL}.

We have been able to overcome this problem only for Maxwellian myxos, where we construct a Lyapunov function 
of the form 
$$
  \h[f] = W_2^{\T}(f,\bar f)^2 + \gamma W_2^{\T}(\bar f,f_\infty)^2 \,,
$$
with $\gamma>0$, and where $\bar f$ denotes the partial equilibrium
\begin{equation}\label{av-angles}
   \bar f(\vp,t) = \rho_+ \delta(\vp - \bar\vp_+(t)) + \rho_- \delta(\vp - \bar\vp_-(t)) \,,\qquad\mbox{with }
   \bar\vp_\pm(t) := \frac{1}{\rho_\pm} \int_{\T_\pm} \vp f(\vp,t) d\vp \,.
\end{equation}
This implies 
$$
  W_2^{\T}(f,\bar f)^2 = \int_{\T_+} (\vp - \bar\vp_+)^2 f \,d\vp + \int_{\T_-} (\vp - \bar\vp_-)^2 f \,d\vp \,,
$$
and
\begin{align}\label{Hrev}
   W_2^{\T}(\bar f,f_\infty)^2 = \rho_+ (\bar\vp_+ - \vp_+)^2 + \rho_- (\bar\vp_- - \vp_+^\downarrow)^2
   = \frac{\rho_+ \rho_-}{\rho_+ + \rho_-} (\bar\vp_+ - \bar\vp_- - \pi)^2\,,
\end{align}
where the second equality is due to the conservation law \eqref{phi-cons}, i.e.,
$$
  \rho_+\bar\vp_+ + \rho_- \bar\vp_- = \rho_+\vp_+ + \rho_- \vp_+^\downarrow \,.
$$
For the time derivative of the first contribution we obtain, similarly to \eqref{HAL-dissipation2}, but now with $b\equiv 1$,
\begin{eqnarray*}
  \frac{d}{dt} W_2^{\T}(f,\bar f)^2 &=&  -\frac{1}{4} \int_{\T_+} \int_{\T_+} f f_*(\vp-\vpa)^2 d\vpa d\vp 
   -\frac{1}{4} \int_{\T_-}\int_{\T_-} f f_*(\vp-\vpa)^2 d\vpa d\vp \\
   && + 2\rho_+ \rho_- (\bar\vp_+ - \bar\vp_- - \pi)^2 \,,
\end{eqnarray*}
where the nonnegative term in the second line results from the reversal collisions. The time derivative of the second contribution is not influenced by alignment collisions:
$$
  \frac{d}{dt} W_2^{\T}(\bar f,f_\infty)^2  = -2\rho_+\rho_-(\bar\vp_+ - \bar\vp_- - \pi)^2 \,,
$$
from which, together with \eqref{Hrev}, exponential decay of $W_2^{\T}(\bar{ f},f_{\infty})^2$, the reversal part of our Lyapunov function follows.  
Finally, the identity
$$
   \int_{\T_\pm} \int_{\T_\pm} f f_*(\vp-\vpa)^2 d\vpa d\vp = 2\rho_\pm \int_{\T_\pm} f (\vp-\bar\vp_\pm)^2 d\vp
$$
implies
\begin{eqnarray*}
   \frac{d\h[f]}{dt} &=&  -\frac{\rho_+}{2} \int_{\T_+} f (\vp-\bar\vp_+)^2 d\vp 
   -\frac{\rho_-}{2} \int_{\T_-} f (\vp-\bar\vp_-)^2 d\vp \\
   && - 2(\gamma-1)\rho_+ \rho_- (\bar\vp_+ - \bar\vp_- - \pi)^2 \le 0\,,
\end{eqnarray*}
for $\gamma\ge 1$. It is easily seen that with the choice $\gamma = 8/7$ we have
\begin{equation}\label{H-dissipation}
   \frac{d\h[f]}{dt} \le -2\lambda \h[f] \,,\qquad\mbox{with } \lambda = \frac{1}{4} \min\{\rho_+,\rho_-\} \,.
\end{equation}

\begin{theorem}
Let $f_I \in L_+^1(\T)$ with ${\rm supp}(f_I) \subset \T_+\cup \T_-$, and let $f$ be a solution of \eqref{hom}. Then
for Maxwellian myxos, i.e. $b(\vp,\vpa)\equiv 1$, there exists $C>0$, such that
$$
   W_2^{\T}(f(\cdot,t),f_\infty) \le C e^{-\lambda t} \,,\qquad \forall\, t\ge 0 \,,
$$
with $f_\infty$ defined in \eqref{Equ} and $\lambda$ as in \eqref{H-dissipation}.
\end{theorem}

\begin{proof}
After using \eqref{H-dissipation}, it only remains to use the triangle inequality for the Wasserstein distance to obtain
$W_2^{\T}(f,f_\infty)^2 \le 2\h[f]$.
\end{proof}

\section{Numerical Simulations}

\paragraph{Discretization:}
The results of the preceding section will be illustrated by numerical simulations of the spatially homogeneous model (\ref{hom}). Discretization in the angle direction is based on an equidistant grid
$$
    \vp_k = \frac{(k-n)\pi}{n} \,,\qquad k = {\color{violet} 0},\ldots,2n \,,
$$
with an even number of grid points, guaranteeing that the grid is invariant under reversal collisions, i.e., with $\vp_k$
also $\vp_k^\downarrow=\vp_{k+n}$ is a grid point. Similarly, only those alignment collisions between discrete angles
will be allowed, which produce post-collisional angles belonging to the grid. This is facilitated by rewriting the 
alignment collision operator \eqref{Q:AL+REV} as 
$$
  Q_{AL}(f,f) = 2\int_{\TAin} b(\tilde\vp,\vpa) (\tilde f f_* - f \tilde f_*)d\vpa \,,
$$
with $\tilde \vp = 2\vp-\vpa$, $\tilde\vp_* = 2\vpa-\vp$, before discretization. Note that in this form mass conservation is
obvious since $b(\tilde\vp,\vpa) = b(\tilde\vp_*,\vp)$, and the grid is invariant under the map 
$(\vp,\vpa)=(\vp_k,\vp_{k_*})\mapsto (\tilde\vp,\tilde\vp_*)=(\vp_{2k-k_*},\vp_{2k_*-k})$. Finally, we always choose $n$
odd to avoid the angle $\pi/2$ between grid angles and, thus, the ambiguity between alignment and reversal collisions.

Solutions of \eqref{hom} are approximated at grid points by
$$
   f^n(t) := (f_1(t),\ldots,f_{2n}(t)) \approx (f(\vp_1,t),\ldots,f(\vp_{2n},t)) \,,
$$
extended periodically by $f_{k+2n}(t) = f_k(t)$. This straightforwardly leads to the discrete model
\begin{align}\label{discrete}
\frac{df_k}{dt} = Q^n(f^n,f^n)_k \,,
\end{align}
with 
$$
   Q^n(f^n,f^n)_k := \frac{2\pi}{n}\sum_{|k_*-k|<n/4} b_{2k-k_*,k_*} (f_{2k-k_*}f_{k_*} - f_k f_{2k_*-k}) 
   + \frac{\pi}{n}\sum_{|k_*-k|>n/2} b_{k,k_*} (f_{k+n}f_{k_*+n} - f_k f_{k_*}) \,,
$$
and $b_{k,k_*}:=b(\vp_k,\vp_{k_*})$. 

For the time discretization the \textit{explicit Euler scheme} is used, such that the total mass is conserved by the
discrete scheme, which has been implemented in \textsc{Matlab}.

\paragraph{Numerical simulations with two group initial conditions:}

Simulations have been carried out with $n=201$ and with the time step $\Delta t = 0.1$. In the first rows of Figures \ref{spIC1}, \ref{spIC2}, \ref{genIC}, 
density is color coded as a function of $\vp$ (vertically) and $t$ (horizontally). The plots in the second rows show snapshots of the distribution function $f$ at different times. 

Although we only provide a proof for Maxwellian myxos, we expect solutions of (\ref{hom}) with initial data satisfying (\ref{spin}) to converge to the equilibrium $f_{\infty}$, given by \eqref{Equ}, also for rod shaped myxobacteria. This conjecture is supported by the simulation results depicted in Figures \ref{spIC1}, \ref{spIC2}.

On the left side of Figure \ref{spIC1} the initial distribution is uniform within both $\T_{+}$ and $\T_{-}$, but with zero mass outside. The equilibrium angle is given by $\vp_+ = \frac{\pi}{2}$. The initial data on the right side are similar, but with no mass in intervals around $\pi/2$ and $-\pi/2$, which again causes $\vp_+=\frac{\pi}{2}$. 

In the left part of Figure \ref{spIC2} the initial data are supported in $\T_+$, therefore excluding reversal collisions. The discretization preserves the mass  conservation in both $\T_+$ and $\T_-$ separately. This is the situation of Lemma \ref{lem:AL} b). The decay estimate for the variance
as $t^{-2}$ (Haff's law) is demonstrated by the left part of Figure \ref{V}. The simulation has also been carried out for Maxwellian myxos, demonstrating
the exponential decay of the variance in this case (Figure \ref{V}, right).
The right part of Figure \ref{spIC2} shows an example, where the average directions $\bar\vp_\pm$ within the groups change significantly. 

\begin{figure}[h!]
	\centering
	\begin{subfigure}{0.48\textwidth} 
		\includegraphics[width=\textwidth]{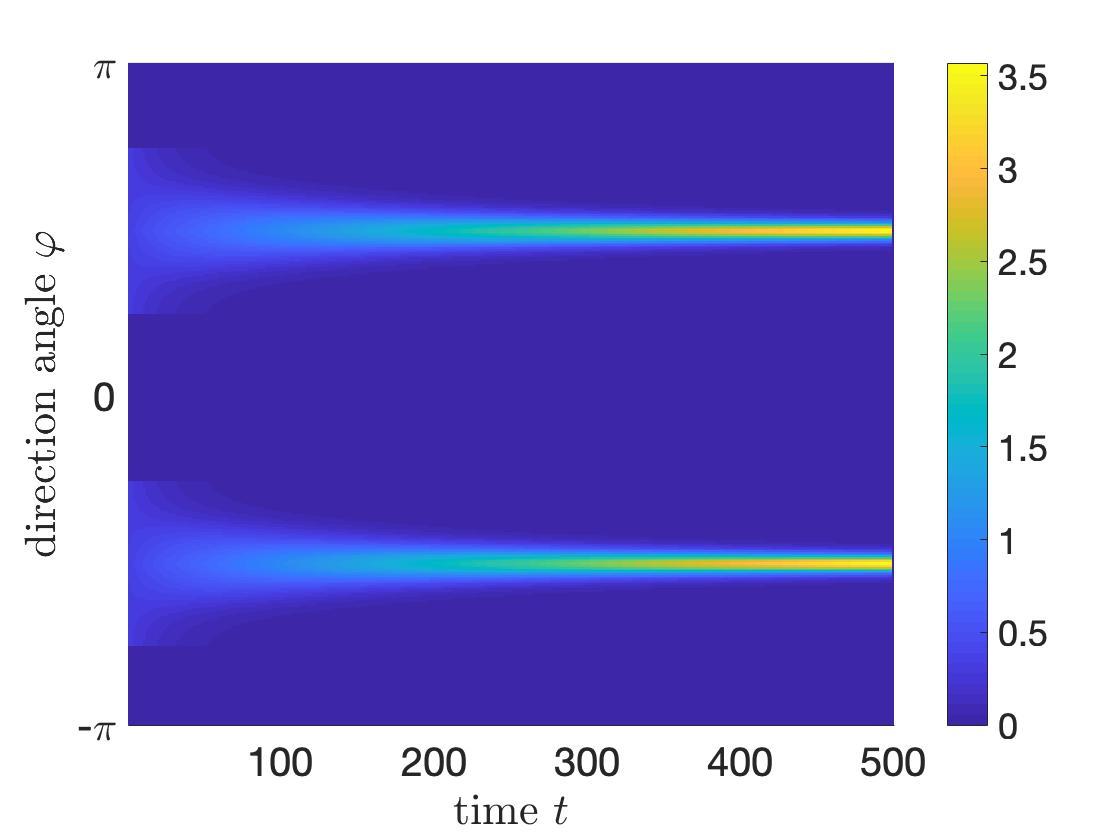}	
	\end{subfigure}
	\hspace{1em} 
	\begin{subfigure}{0.48\textwidth} 
		\includegraphics[width=\textwidth]{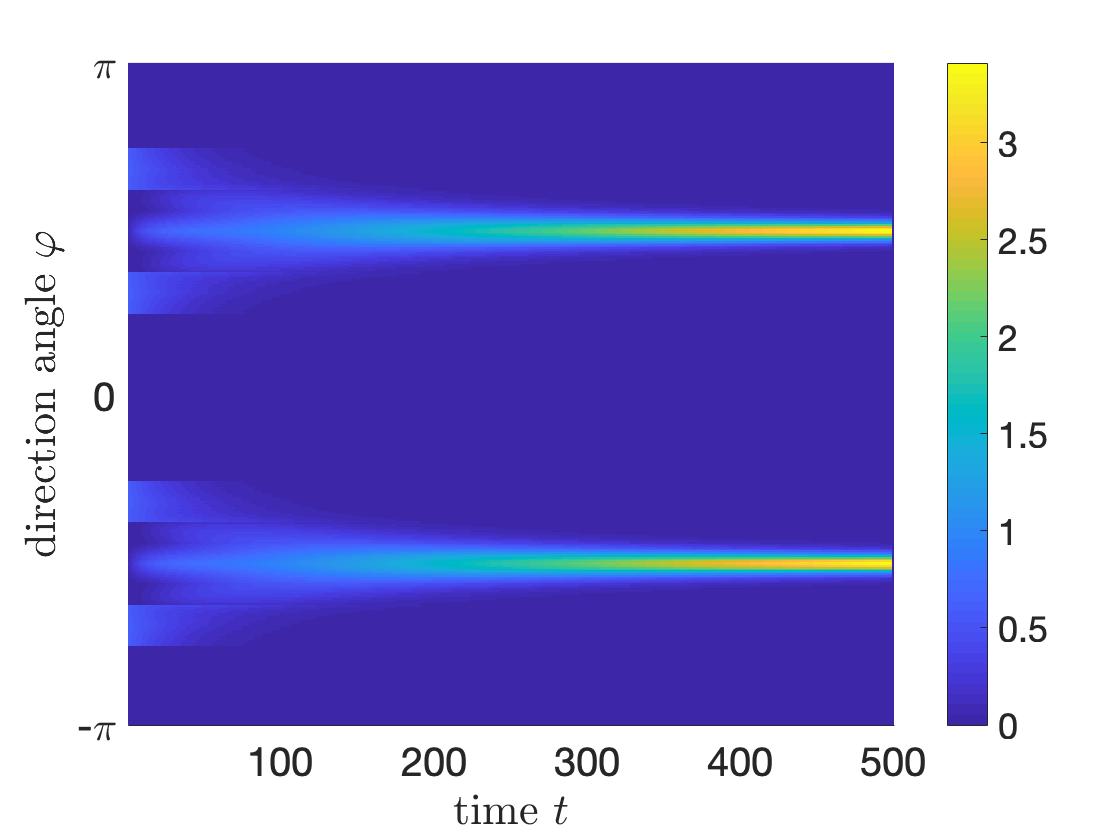}
	\end{subfigure}
	\\
	\begin{subfigure}{0.48\textwidth} 
		\includegraphics[width=\textwidth]{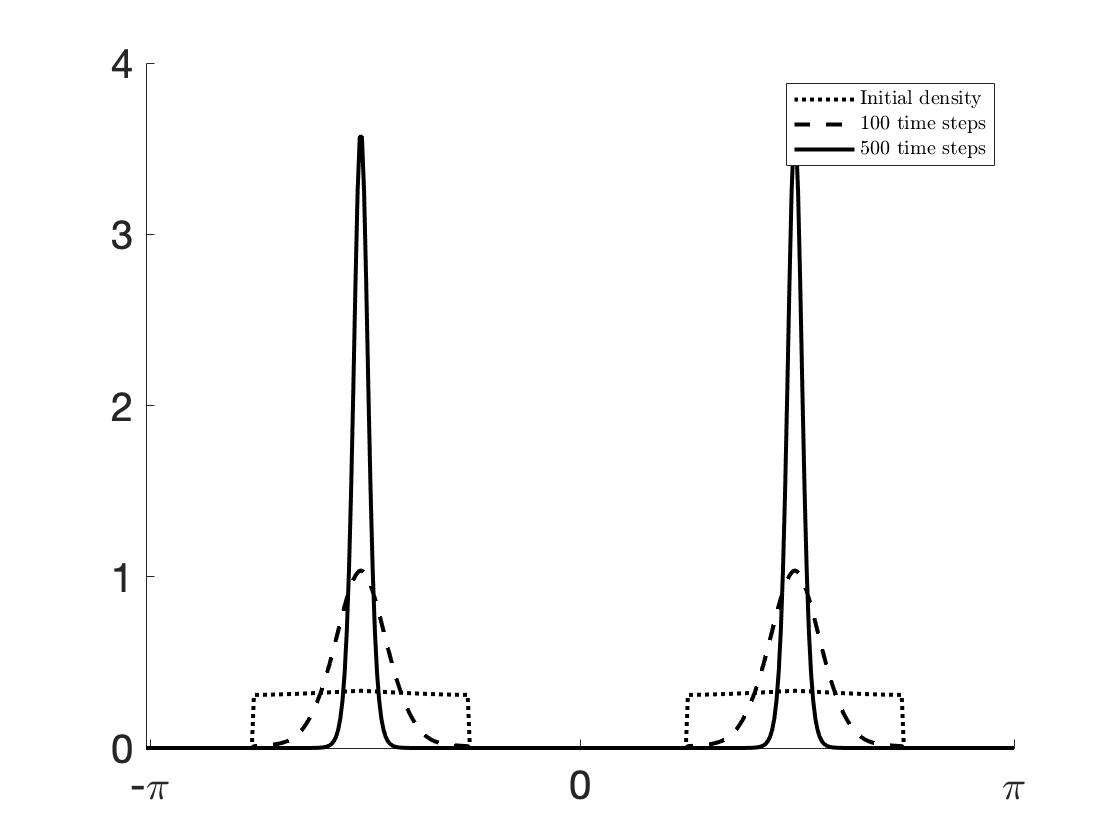}	
	\end{subfigure}
	\hspace{1em} 
	\begin{subfigure}{0.48\textwidth} 
		\includegraphics[width=\textwidth]{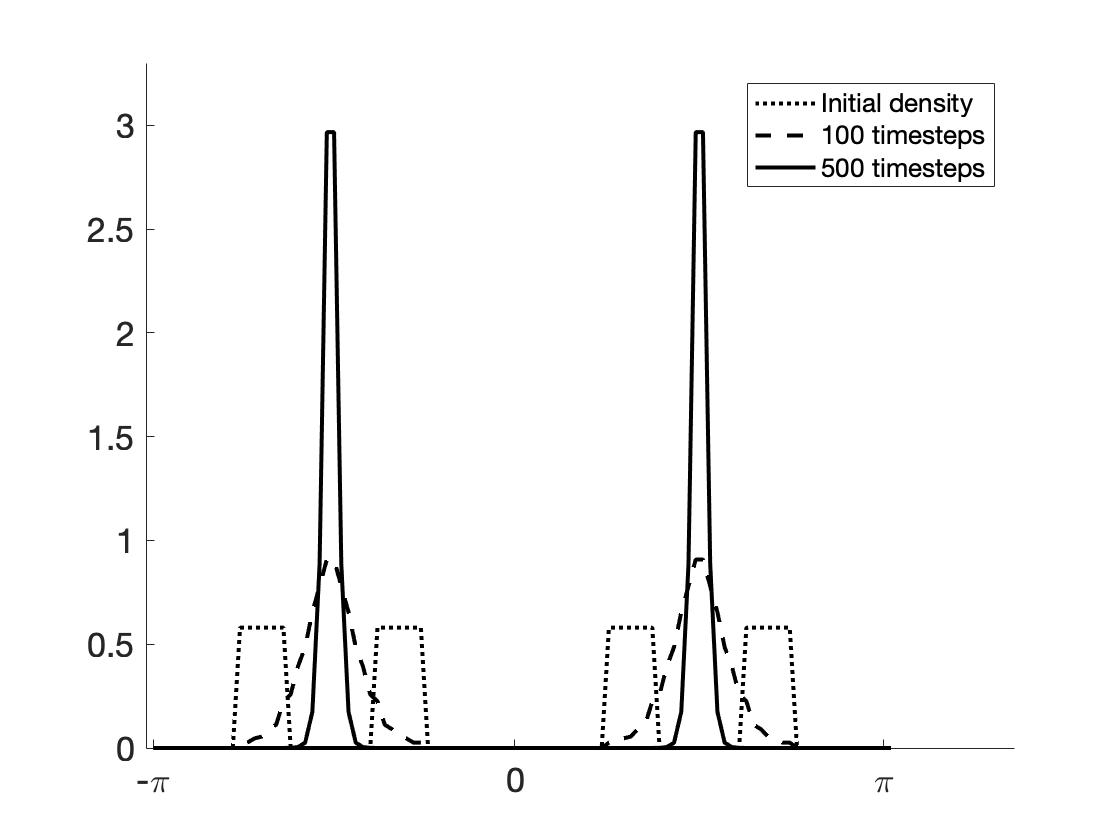}
	\end{subfigure}
	
	\caption{Two group initial conditions with the same mass in $\T_{+}$ and $\T_{-}$; rod shaped bacteria. \emph{Left:} uniform distributions within $\T_{+}$ and $\T_{-}$. \emph{Right:} vacuum around $\pm\pi/2$.}
	\label{spIC1}
\end{figure}

\begin{figure}[h!]
	\centering
	\begin{subfigure}{0.48\textwidth} 
		\includegraphics[width=\textwidth]{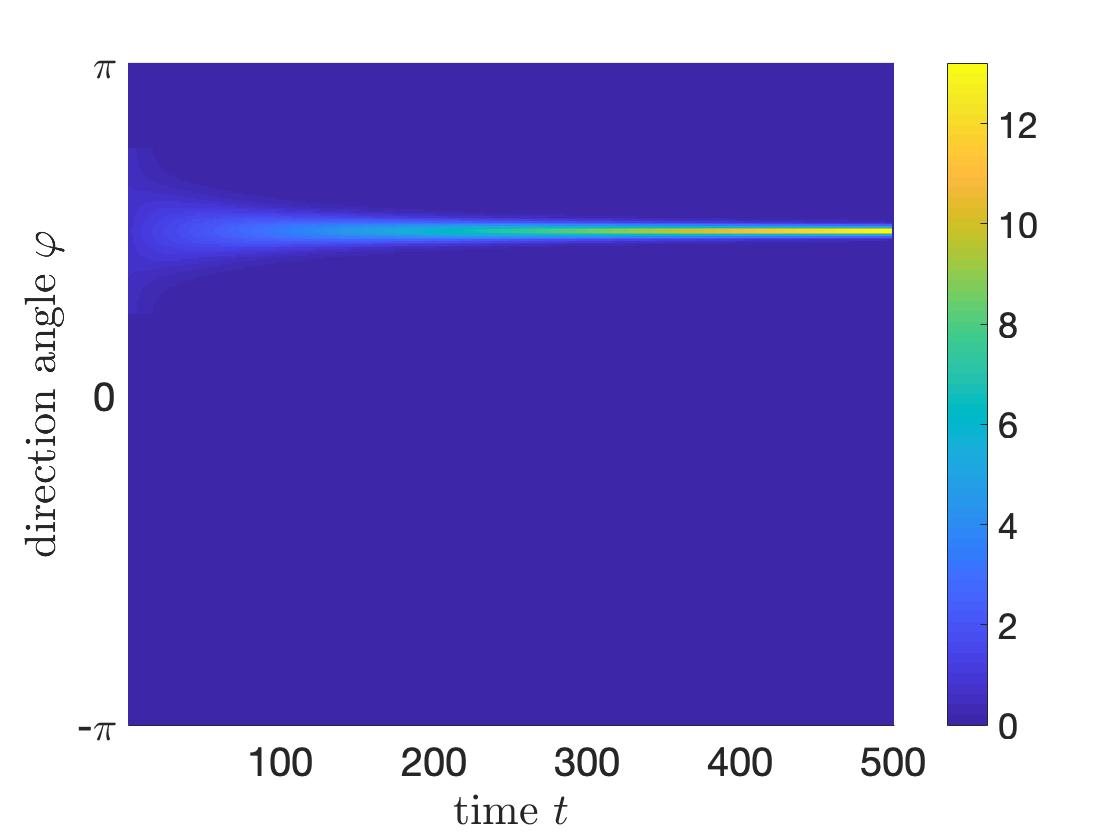}	
	\end{subfigure}
	\hspace{1em} 
	\begin{subfigure}{0.48\textwidth} 
		\includegraphics[width=\textwidth]{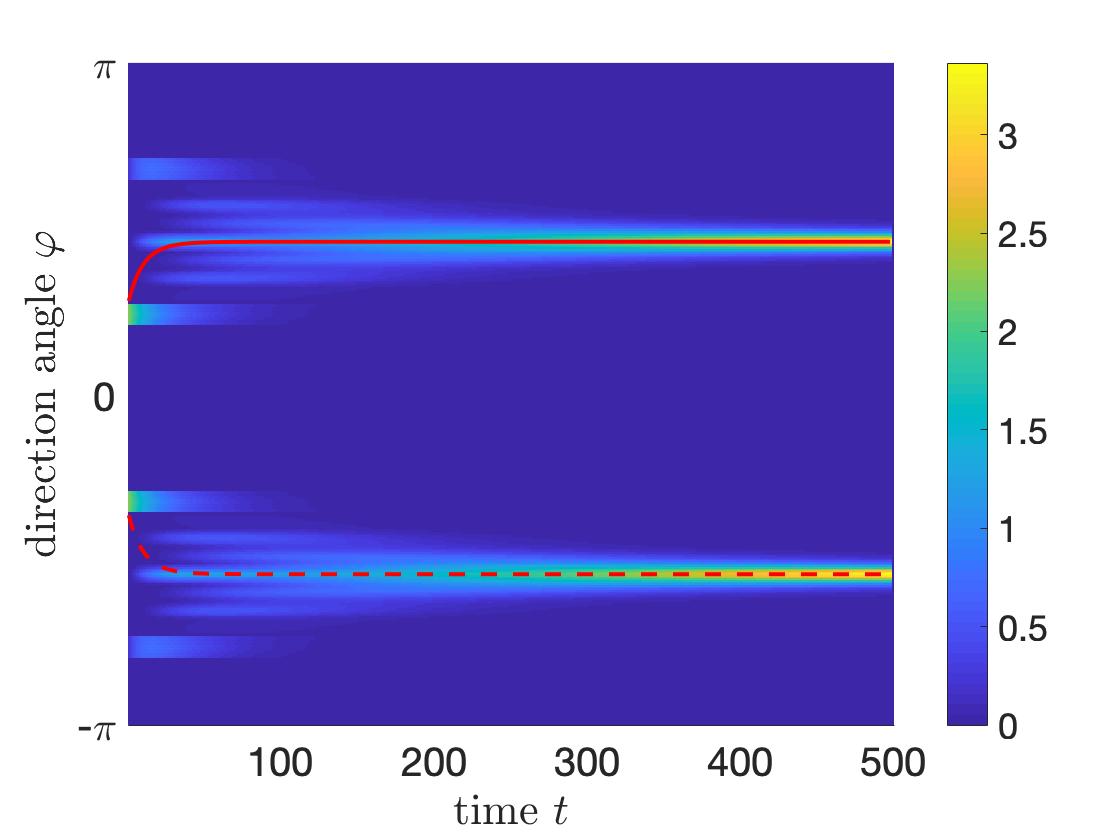}
	\end{subfigure}
	\\
	\begin{subfigure}{0.48\textwidth} 
		\includegraphics[width=\textwidth]{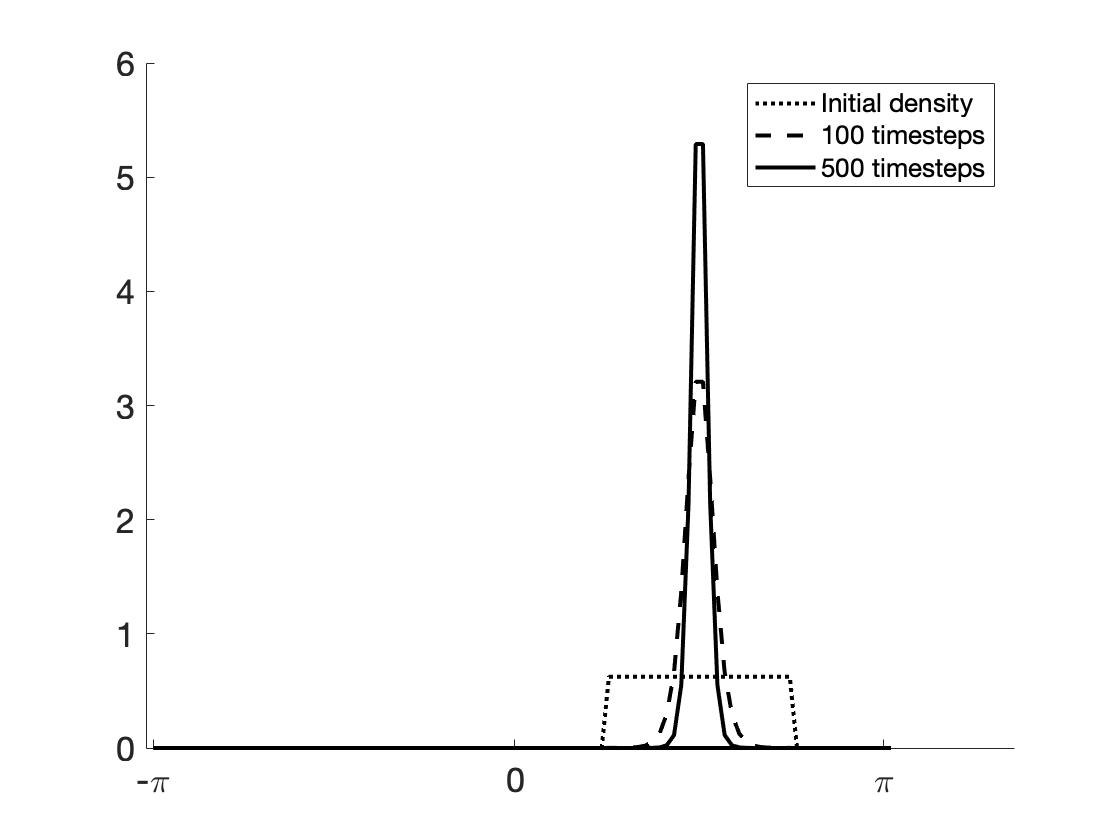}	
	\end{subfigure}
	\hspace{1em} 
	\begin{subfigure}{0.48\textwidth} 
		\includegraphics[width=\textwidth]{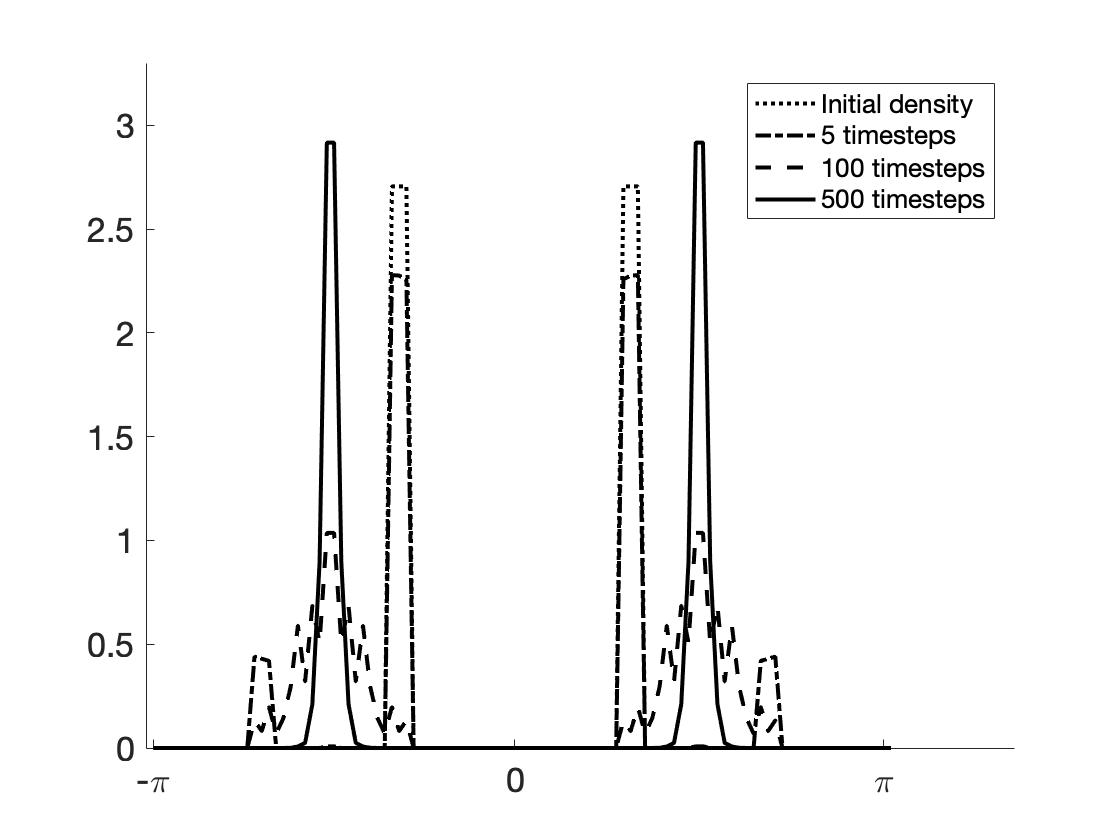}
	\end{subfigure}
	\caption{\emph{Left:} initial condition with uniform distribution in $\T_{+}$ and vacuum everywhere else. \emph{Right:}
		initially two concentrated patches at a distance somewhat bigger than $\pi/2$ (yellow at the left end). Outer stripes 
		created by reversal, then fill-in by alignment, followed by concentration towards opposite directions. The mean angles
		$\bar\vp_+$ (red line) and $\bar\vp_-$ (dotted red line) in the two groups change significantly.}
	\label{spIC2}
\end{figure}

\begin{figure}[h!]
	\centering
	\begin{subfigure}{0.48\textwidth} 
		\includegraphics[width=\textwidth]{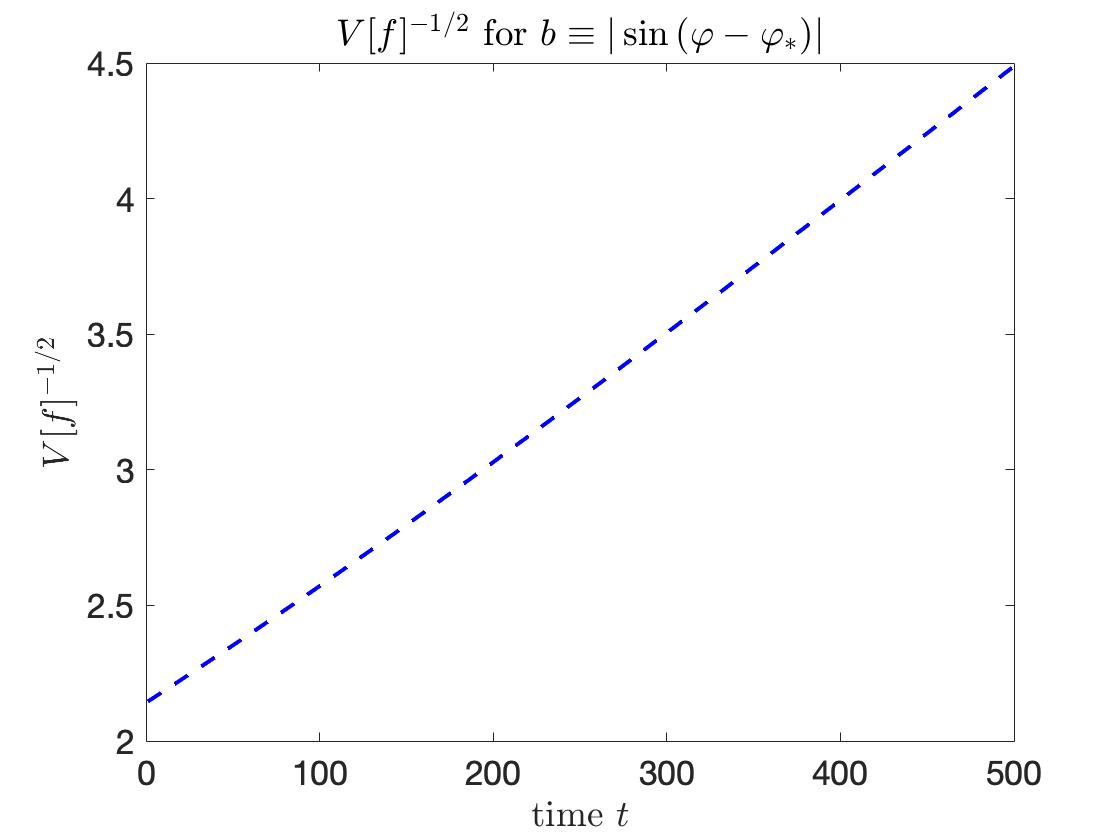}
	\end{subfigure}
	\hspace{1em}
	\begin{subfigure}{0.48\textwidth} 
		\includegraphics[width=\textwidth]{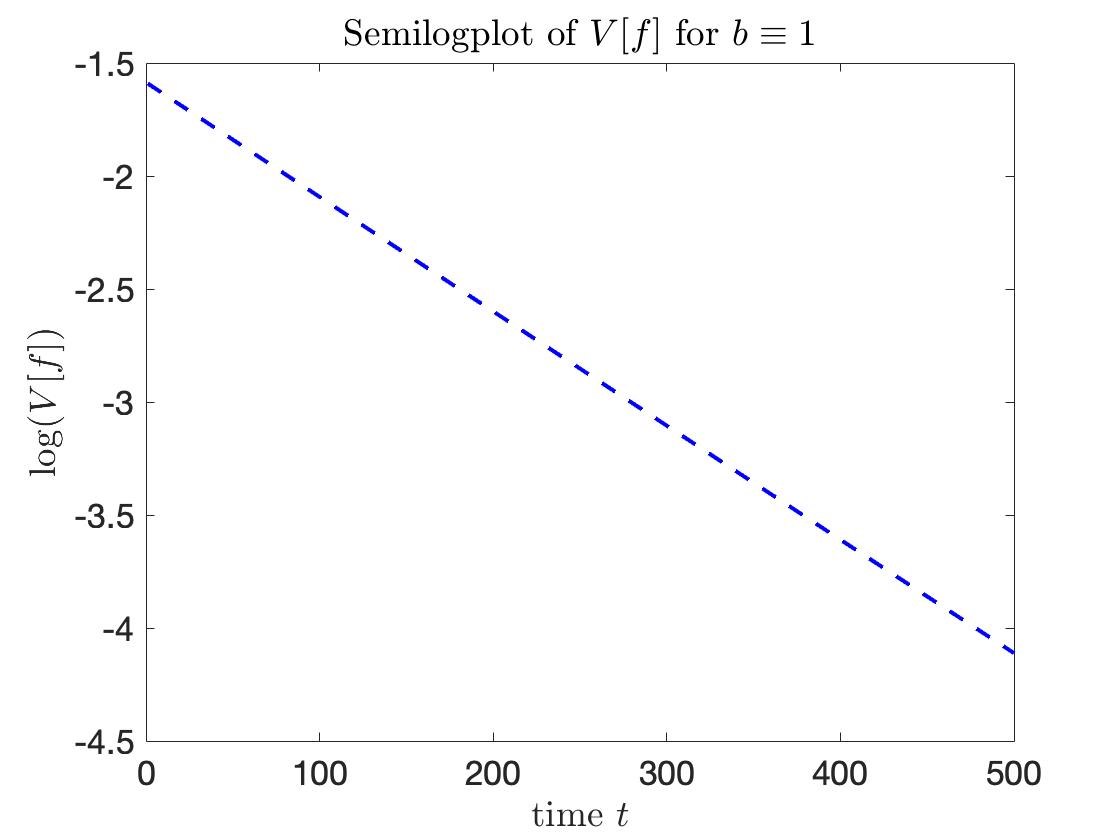}
	\end{subfigure}
	\caption{\emph{Left:} The evolution of the inverse square root of the variance $V[f]$ from the simulation depicted on the left side of Figure \ref{spIC2}, supporting the validity of Haff's law for rod shaped myxos. \emph{Right:} Semi-log plot of $V[f]$ for a simulation with the same initial data, 
	but for Maxwellian myxos, demonstrating exponential decay to equilibrium as shown in Lemma \ref{lem:AL} a).}
	\label{V}
\end{figure}


\paragraph{Instability of constant steady states:}
In Figure \ref{genIC} we consider small perturbations of a constant steady state. On the left side we start with a
random perturbation and see mass concentrating at unpredictable directions $\vp_+$ and $\vp_+^\downarrow$. On the right side we considered a perturbation at one random point $\hat\vp$. We see convergence to $f_\infty$, with equilibrium angle $\vp_+=\hat\vp$. Both simulations illustrate instability of the uniform distribution on $\T$.

\begin{figure}[h!]
	\centering
	\begin{subfigure}{0.45\textwidth} 
		\includegraphics[width=\textwidth]{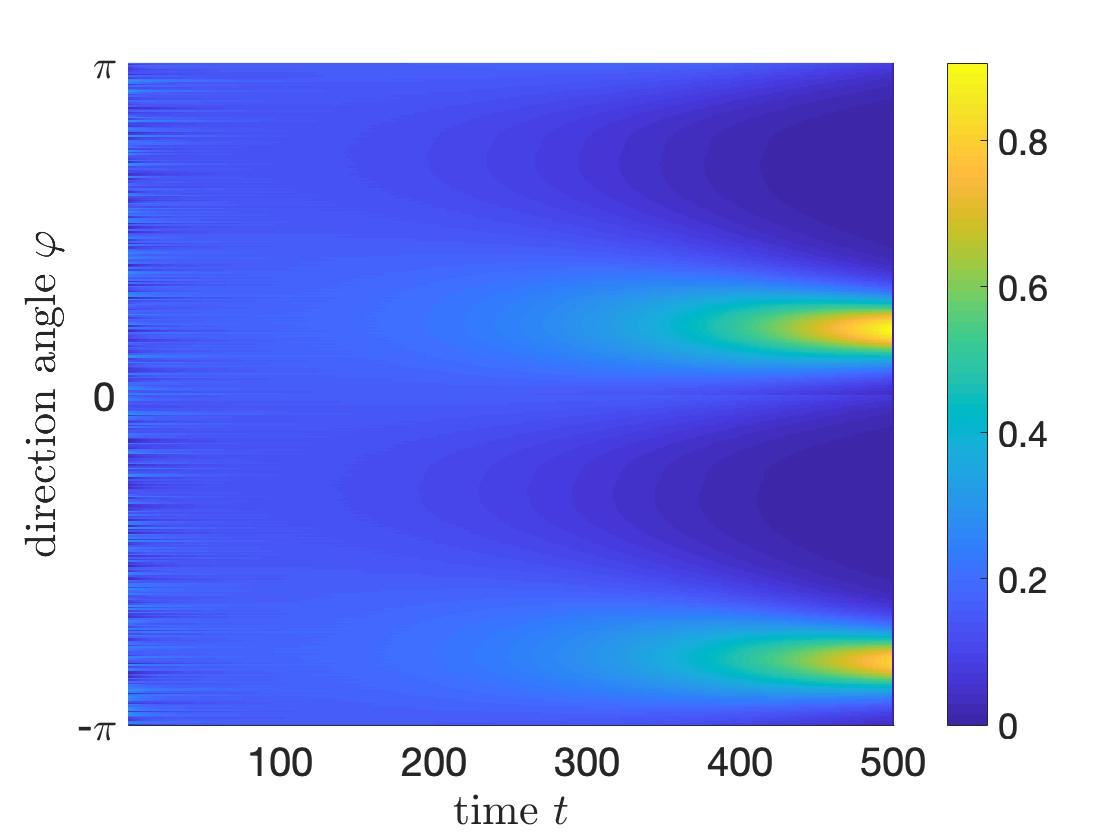}	
	\end{subfigure}
	\hspace{2em} 
	\begin{subfigure}{0.45\textwidth} 
		\includegraphics[width=\textwidth]{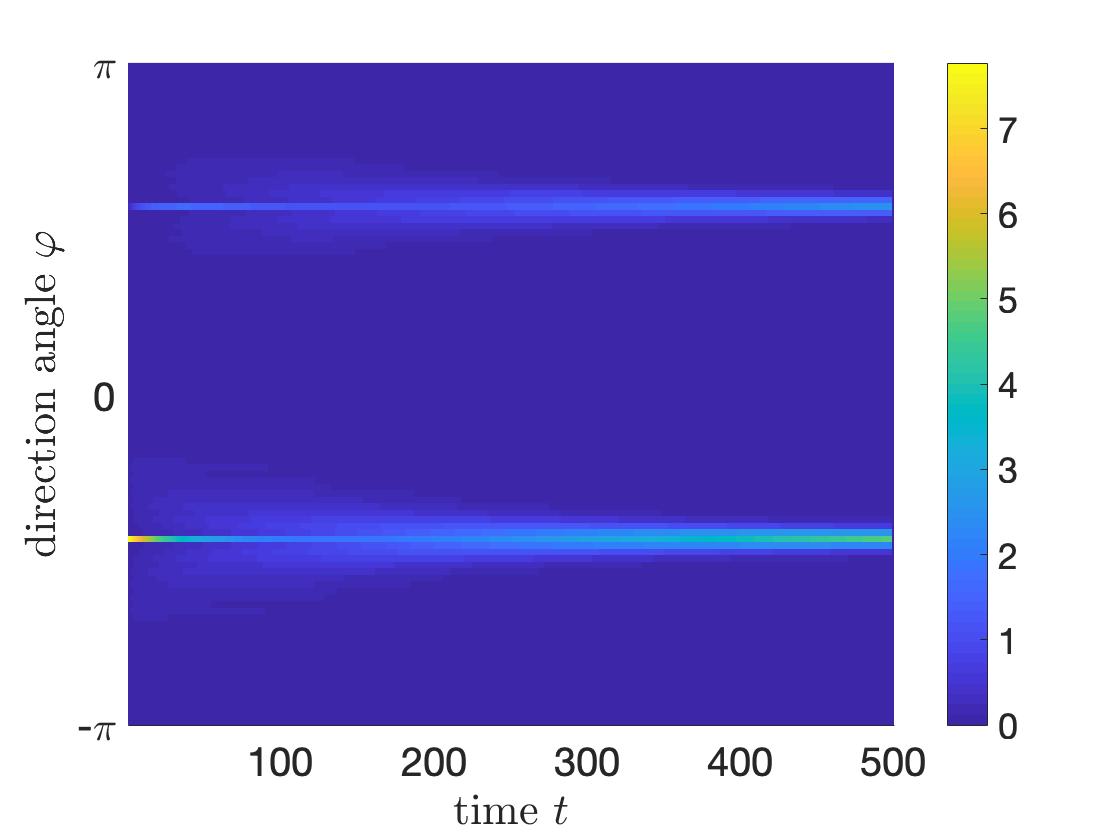}
	\end{subfigure}
	\\
	\begin{subfigure}{0.45\textwidth} 
		\includegraphics[width=\textwidth]{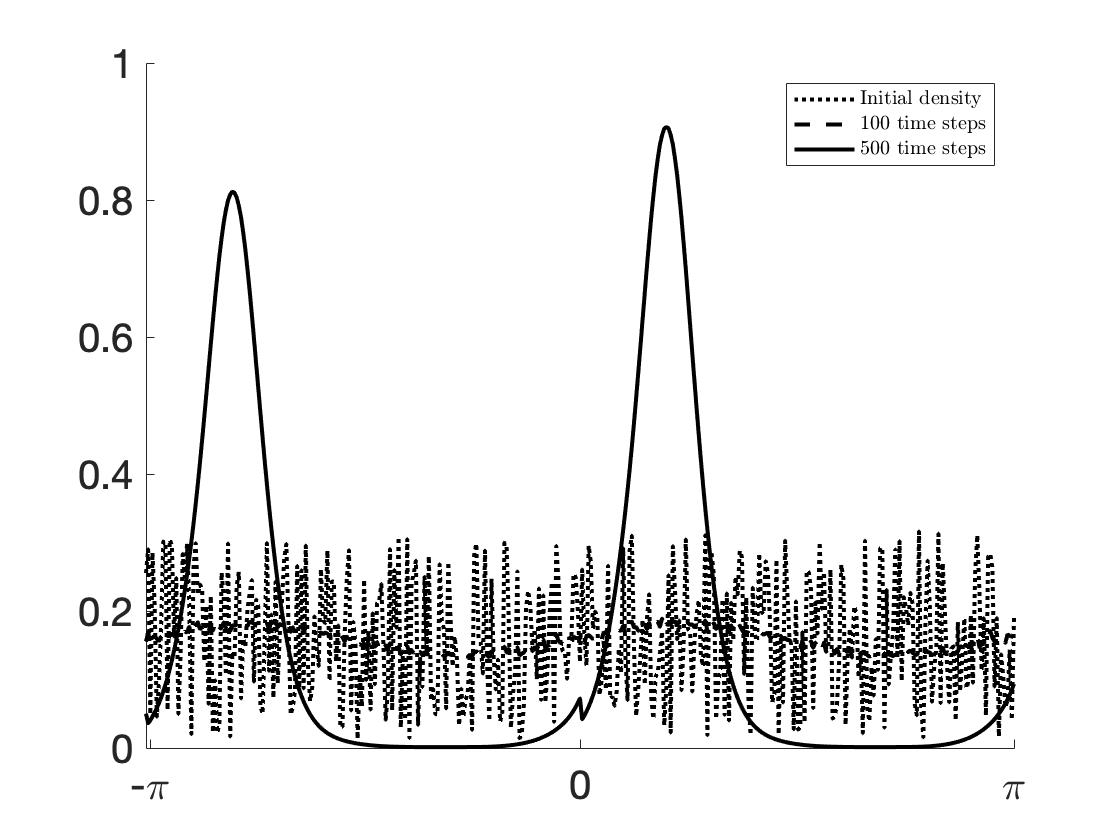}	
	\end{subfigure}
	\hspace{2em} 
	\begin{subfigure}{0.45\textwidth} 
		\includegraphics[width=\textwidth]{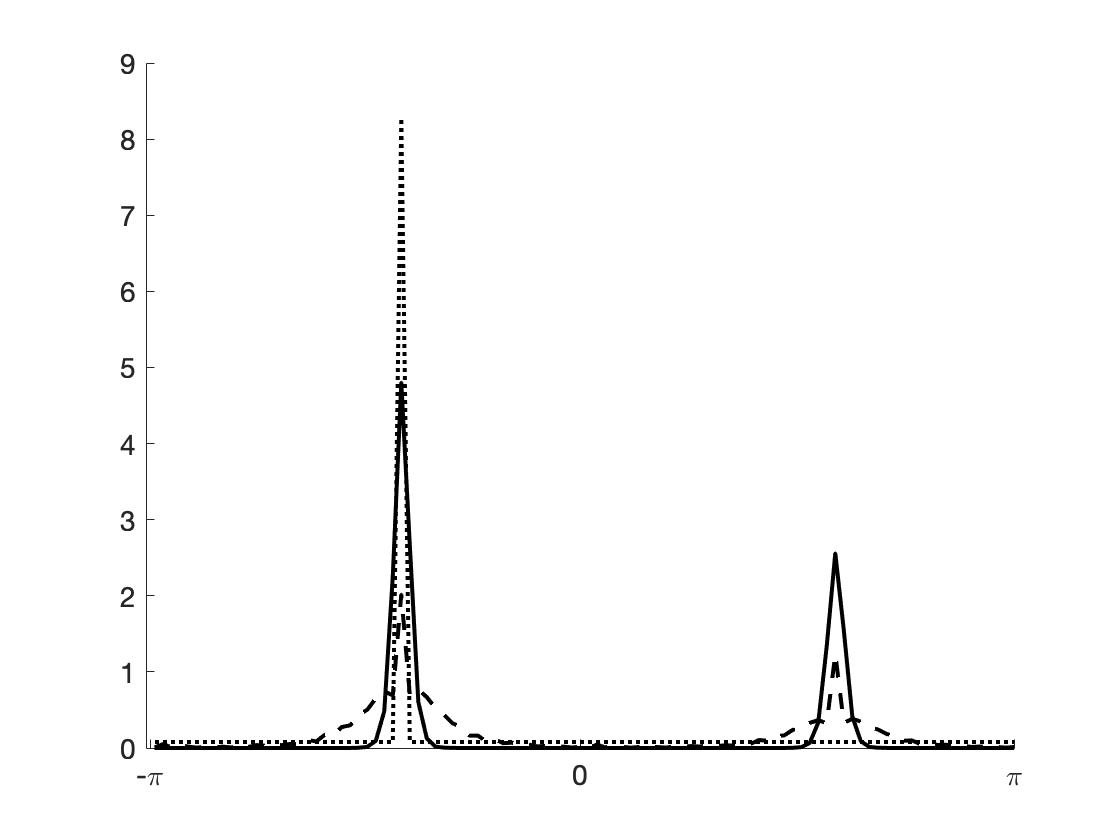}
	\end{subfigure}
	\caption{Instability of constant positive steady states. \emph{Left:} random initial perturbation, leading to an unpredictable
	equilibrium direction. \emph{Right:} initial perturbation at one direction, which eventually becomes the equilibrium 
	direction. Note that this differs from the simulations in Figure \ref{spIC2}, left, by the fact that a positive state is perturbed, and therefore 
	reversal collisions are active.}
	\label{genIC}
\end{figure}

\section{Formal macroscopic limit}

For the simulation of spatial pattern formation phenomena kinetic transport models pose significant numerical challenges and contain often unnecessary 
information on microscopic lengths and time scales. Therefore such simulations are often based on macroscopic models. For myxobacteria colonies
macroscopic models have been formulated both by a direct continuum approach \cite{igoshin3,igoshin,igoshin2,degond3} and based on microscopic or kinetic 
descriptions \cite{baskaran2,bertin,degond}. This section is concerned with the formal macroscopic limit of the kinetic model \eqref{myxo}
to demonstrate which features of other models are reproduced. Similarities can also be found with models for the interaction of microtubules by motor proteins
\cite{aranson} and for granular gases assuming nonelastic collisions \cite{bobylev1,jabin,toscani}. In the latter case the macroscopic limit is often combined
with the assumption of weakly inelastic collisions, leading to an energy balance equation describing the cooling of the gas \cite{bobylev1,toscani}. Since the 
model \eqref{myxo} corresponds to the other extreme of {\em sticky particles}, the macroscopic limit already involves the passage to zero temperature.

We investigate the behavior at macroscopic position and time scales by introducing the rescaling $x\to \frac{x}{\ve}$, 
$t\to\frac{t}{\ve}$, with a Knudsen number $\ve\ll 1$ in \eqref{myxo}:
$$
  \pa_t f^\ve + \omega \cdot \nabla_x f^\ve = \frac{1}{\ve}Q(f^\ve,f^\ve) \,.
$$
Formally, the convergence $f^\ve\to f$ as $\ve\to 0$ implies, by \eqref{Equ},
$$
   f(x,\vp,t) = \rho_+(x,t)\delta\bigl(\vp-\vp_+(x,t)\bigr) + \rho_-(x,t)\delta\left(\vp-\vp_+(x,t)^\downarrow\right) \,.
$$
In Section \ref{sec:3} we have seen that in general the collision operator only allows for two independent collision
invariants $\psi(\vp)=1$ and $\psi(\vp)=\vp$, providing only two conservation laws 
$$
   \pa_t \int_{\T} f \psi\,d\vp + \nabla_x\cdot \int_{\T} \omega f \psi\,d\vp = 0 \,,
$$
for the three unknowns $\rho_+$, $\rho_-$, and $\vp_+$. However, assuming two group initial data (see again Section 
\ref{sec:3}), the mass within the group is a third conserved quantity, closing the macroscopic limit system:
\begin{eqnarray*}
  && \pa_t \rho_+ + \nabla_x\cdot(\rho_+ \omega(\vp_+)) = 0 \,,\\
  && \pa_t \rho_- - \nabla_x\cdot(\rho_- \omega(\vp_+)) = 0 \,,\\
  && \pa_t ((\rho_+ + \rho_-)\vp_+) + \nabla_x\cdot( (\rho_+ -\rho_-)\vp_+ \omega(\vp_+)) = 0 \,.
\end{eqnarray*}
Expanding the derivatives, it can also be written as
\begin{eqnarray}\label{macro}
  && \pa_t \rho_+ + \omega\cdot\nabla_x\rho_+ + \rho_+ \omega^\bot\cdot\nabla_x \vp_+ = 0 \,, \notag\\
  && \pa_t \rho_- - \omega\cdot\nabla_x\rho_- - \rho_- \omega^\bot\cdot\nabla_x \vp_+ = 0 \,,\\
  && \pa_t \vp_+ + \frac{\rho_+ -\rho_-}{\rho_+ + \rho_-}\omega\cdot\nabla_x\vp_+ = 0 \,, \notag
\end{eqnarray}
showing that for $\rho_+,\rho_- > 0$ the system is strictly hyperbolic with characteristic velocities $\omega$,
$-\omega$, $\frac{\rho_+ -\rho_-}{\rho_+ + \rho_-}\omega$. Although the system is nonlinear, all three characteristic 
fields are linearly degenerate.  On the other hand, the special case $\rho_-=0$ leads to 
\begin{eqnarray*}
  && \pa_t \rho_+ + \omega\cdot\nabla_x\rho_+ + \rho_+ \omega^\bot\cdot\nabla_x \vp_+ = 0 \,,\\
  && \pa_t \vp_+ + \omega\cdot\nabla_x\vp_+ = 0 \,,
\end{eqnarray*}
a non-strictly hyperbolic system with the same structure as the equations for pressureless gas dynamics, derived as macroscopic limit of the dissipative Boltzmann equation \cite{jabin}. \\
Furthermore, comparing the equation for $\vp_+$ in (\ref{macro}) with the macroscopic one for the equilibrium angle in \cite{degond}, we see that they only differ by a pressure term proportional to $(\rho_+ -\rho_-) \omega^{\perp} \cdot \nabla_x \vp_+$ not occurring in our case. Considering the limit of vanishing diffusion in \cite{degond} this term vanishes, which reveals the fact that the two different microscopic models provide the same macroscopic equations. 
The models in \cite{baskaran2} and \cite{bertin} are quite different. They consider only one macroscopic density, coupled with a nematic polarization 
vector and an order parameter in \cite{baskaran2}, and with a mean velocity with variable speed in \cite{bertin}.

\newpage 
\section*{Acknowledgments}
This work has been supported by the Austrian Science Fund (FWF) project F65 \textit{Taming Complexity in Partial Differential Systems}. C.S. acknowledges support by the Austrian Science Fund (grant no. W1245), by the Fondation Sciences Math\'ematiques de Paris, and by Paris Science et Lettres. S.H.  acknowledges support via FWF project T-764. The authors also acknowledge the comments of an anonymous referee, who pointed out a significant 
number of references.

\end{document}